\documentclass[a4paper,10pt,reqno,oneside]{amsart}
\usepackage{amsmath,amsthm,amssymb,enumerate,mathtools,stmaryrd}
\usepackage[latin1]{inputenc}
\usepackage{bbm}
\usepackage{esint}
\usepackage{multirow}

\usepackage{euscript,mathrsfs}
\usepackage{color}
\usepackage[left=3.2cm,right=3.2cm,top=3.2cm,bottom=3cm]{geometry}
\usepackage[colorlinks=true, linktocpage=true, linkcolor=red!70!black, citecolor=green!50!black, urlcolor=black]{hyperref}
\allowdisplaybreaks
\usepackage{tikz}

\usepackage{enumitem}
\setenumerate{label={\rm (\alph{*})}}

\numberwithin{equation}{section}

\newcommand{\dashint}{\fint}

\newcommand{\R}{\mathbb R}
\newcommand{\N}{\mathbb N}
\DeclareMathOperator{\tr}{tr}
\newcommand{\ntimes}[1]{\times_{#1}}
\newcommand{\nMat}[2]{\left\llbracket#1\right\rrbracket_{\ntimes{#2}}}
\renewcommand{\div}{\operatorname{div}}
\newcommand{\Div}{\operatorname{Div}}
\newcommand{\Cof}{\operatorname{Cof}}
\newcommand{\Anti}{\operatorname{Anti}}
\newcommand{\dif}{\operatorname{d}\!}
\newcommand{\lebe}{\operatorname{L}}
\newcommand{\sobo}{\operatorname{W}}
\newcommand{\imag}{\operatorname{i}}
\newcommand{\hold}{\operatorname{C}}
\newcommand{\ball}{\operatorname{B}}
\newcommand{\A}{\mathbb{A}}
\newcommand{\D}{\operatorname{D}\!}
\newcommand{\spt}{\operatorname{spt}}
\newcommand{\bmo}{\operatorname{BMO}}
\newcommand{\curl}{\operatorname{curl}}
\newcommand{\sym}{\operatorname{sym}}
\renewcommand{\skew}{\operatorname{skew}}
\newcommand{\dev}{\operatorname{dev}}
\newcommand{\norm}[1]{\left\lVert#1\right\rVert}
\renewcommand{\Re}{\mathrm{Re}}
\renewcommand{\Im}{\mathrm{Im}}
\newcommand{\Curl}{\operatorname{Curl}}
\newcommand{\Atimes}[2]{\,#1\otimes_{\mathbb{A}}#2}
\newcommand{\eI}[1]{\mathbf{e}_{#1}}
\newcommand{\Q}{\mathfrak{G}}
\newcommand{\f}{\mathfrak{f}}

\newcounter{alphasect}
\def\alphainsection{0}

\let\oldsection=\section
\def\section{%
  \ifnum\alphainsection=1%
    \addtocounter{alphasect}{1}
  \fi%
\oldsection}%
\renewcommand\thesection{%
  \ifnum\alphainsection=1%
    \Alph{alphasect}%
  \else%
    \arabic{section}%
  \fi%
}%

\newenvironment{alphasection}{%
  \ifnum\alphainsection=1%
    \errhelp={Let other blocks end at the beginning of the next block.}
    \errmessage{Nested Alpha section not allowed}
  \fi%
  \setcounter{alphasect}{0}
  \def\alphainsection{1}
}{%
  \setcounter{alphasect}{0}
  \def\alphainsection{0}
}%

\makeatletter
\@namedef{subjclassname@2020}{\textup{2020} Mathematics Subject Classification}
\makeatother
\newtheorem{theorem}{Theorem}[section]
\newtheorem{lemma}[theorem]{Lemma}
\newtheorem{proposition}[theorem]{Proposition}

\theoremstyle{remark}
\newtheorem{remark}[theorem]{Remark}
\newtheorem{example}[theorem]{Example}

\newtheorem{OQ}[theorem]{Open Question}

\definecolor{fg}{RGB}{34,139,34}  
\usepackage{libertine}

\begin{document}
\begin{tikzpicture}[remember picture, overlay]
 \node [xshift=-1cm,yshift=15cm,rotate=-90] at (current page.south east)
 {The final publication appeared in Calculus of Variations and Partial Differential Equations (2023), doi: \href{https://doi.org/10.1007/s00526-023-02522-6}{10.1007/s00526-023-02522-6   
}.
 };
\end{tikzpicture}

\numberwithin{equation}{section}

\title[KMS-inequalities in all dimensions]{Optimal incompatible Korn-Maxwell-Sobolev \\ inequalities in all dimensions}
\author[F.~Gmeineder]{Franz Gmeineder}
\address[F. Gmeineder]{University of Konstanz, Department of Mathematics and Statistics, Universit\"{a}tsstra\ss e 10, 78457 Konstanz, Germany.}
\author[P.~Lewintan]{Peter Lewintan} 
\author[P.~Neff]{Patrizio Neff}
\address[P. Lewintan \& P. Neff]{Faculty of Mathematics, University of Duisburg-Essen, Thea-Leymann-Stra\ss e 9,
45127 Essen, Germany.}
\date\today

\maketitle

\begin{abstract}
We characterise all linear maps $\mathscr{A}\colon\R^{n\times n}\to\R^{n\times n}$ such that, for $1\leq p<n$, 
\begin{align*}
\norm{P}_{\lebe^{p^{*}}(\R^{n})}\leq c\,\Big(\norm{\mathscr{A}[P]}_{\lebe^{p^{*}}(\R^{n})}+\norm{\Curl P}_{\lebe^{p}(\R^{n})} \Big)
\end{align*}
holds for all compactly supported $P\in\hold_{c}^{\infty}(\R^{n};\R^{n\times n})$, where $\Curl P$ displays the matrix curl. Being applicable to incompatible, that is, non-gradient matrix fields as well, such inequalities generalise the usual Korn-type inequalities used e.g. in linear elasticity. Different from previous contributions, the results gathered in this paper are applicable to all dimensions and optimal. This particularly necessitates the distinction of different {combinations} between the ellipticities of $\mathscr{A}$, the integrability $p$ and the underlying space dimensions $n$, especially requiring a finer analysis in the two-dimensional situation. 
\end{abstract}
\section{Introduction}
One of the most fundamental tools in (linear) elasticity or fluid mechanics are \emph{Korn-type inequalities}. Such inequalities are pivotal for coercive estimates, leading to well-posedness and regularity results in spaces of weakly differentiable functions; see  \cite{Ciarlet, CFI,ContiGmeineder,Friedrichs,FuchsSeregin,Gm21,Gobert,Horgan,Korn,Korn1,Malek,PayneWeinberger} for an incomplete list. In their most basic form they assert that for each $n\geq 2$ and each $1<{q}<\infty$ there exists a constant $c=c({q},n)>0$ such that 
\begin{align}\label{eq:Korn1}
\norm{\D u}_{\lebe^{{q}}(\R^{n})}\leq c\norm{\varepsilon(u)}_{\lebe^{{q}}(\R^{n})}=c\norm{\D u+\D u^{\top}}_{\lebe^{{q}}(\R^{n})}
\end{align}
holds for all $u\in\hold_{c}^{\infty}(\R^{n};\R^{n})$. Here, $\varepsilon(u)=\sym \D u$ is the symmetric part of the gradient. Within linearised elasticity, where $\varepsilon(u)$ takes the role of the infinitesimal strain tensor for some displacement $u\colon\Omega\to\R^{n}$, variants of~inequality \eqref{eq:Korn1} imply  the existence of minimisers for elastic energies
\begin{align}\label{eq:mainfunctional}
u\mapsto \int_{\Omega}W(\sym \D u)\dif x
\end{align}
in certain subsets of Sobolev spaces $\sobo^{1,{q}}(\Omega;\R^{n})$ provided the elastic energy density $W$ satisfies suitable growth and semiconvexity assumptions, see e.g. \textsc{Fonseca \& M\"{u}ller} \cite{FonsecaMueller}. Variants of~\eqref{eq:Korn1} also prove instrumental in the study of (in)compressible fluid flows  \cite{Breit,Feireisl,Malek} or in the momentum constraint equations from general relativity and trace-free infinitesimal strain measures \cite{BauerNeffPaulyStarke,Dain,FuchsSchirra,Lewintan3,Lewintan4,Reshetnyak,Schirra}.

 Originally,~\eqref{eq:Korn1} was derived by \textsc{Korn} \cite{Korn} in the $\lebe^{2}$-setting and later on generalised to all $1<{q}<\infty$. Inequality~\eqref{eq:Korn1} is non-trivial because it strongly relies on gradients not being arbitrary matrix fields; note that there is no constant $c>0$ such that 
\begin{align}\label{eq:Korn2}
\norm{P}_{\lebe^{{q}}(\R^{n})}\leq c\norm{\sym P}_{\lebe^{{q}}(\R^{n})},\qquad P\in\hold_{c}^{\infty}(\R^{n};\R^{n\times n})
\end{align}
holds. As such, \eqref{eq:Korn2} fails since general matrix fields need not be $\Curl$-free. It is therefore natural to consider variants of \eqref{eq:Korn2} that quantify the lack of $\curl$-freeness or \emph{incompatibility}. Such inequalities prove crucial in view of infinitesimal (strain-)gradient plasticity, where functionals typically involve integrands $\D u\mapsto W(\sym e)+|\sym P|^2+V(\Curl P)$ based on the additive decomposition of the displacement gradient $\D u$ into incompatible elastic and irreversible plastic parts: $\D u=e+P$ and $\sym e$ representing a measure of (infinitesimal) elastic strain while $\sym P$ quantifies the plastic strain; see \textsc{Garroni} et al. \cite{GLP}, \textsc{M\"{u}ller} et al. \cite{Lewintan2,MulScaZep} and \textsc{Neff} et al. \cite{EbobisseHacklNeff,Muench,NeffGradient,NeffKnees,NeffPlastic,NesenenkoNeff} for related models. Another field of application are generalised continuum models, e.g. the relaxed micromorphic model, cf.~  \cite{NeffRMM,Neffunifying,SkyNeff}. A key tool in the treatment of such problems are the incompatible \emph{Korn-Maxwell-Sobolev inequalities} which, in the context of~\eqref{eq:Korn2}, read as
\begin{align}\label{eq:Korn3}
\norm{P}_{\lebe^{q}(\R^{n})}\leq c\big(\norm{\sym P}_{\lebe^{q}(\R^{n})}+\norm{\Curl P}_{\lebe^{p}(\R^{n})}\big),\qquad P\in\hold_{c}^{\infty}(\R^{n};\R^{n\times n}), 
\end{align}
where the reader is referred to Section~\ref{sec:notation} for the definition of the $n$-dimensional matrix curl operator. For brevity, we simply speak of \emph{KMS-inequalities}. Scaling directly determines $q$ in terms of $p$ (or vice versa), e.g. leading to the Sobolev conjugate $q=\frac{n\,p}{n-p}$ if $1\leq p<n$. Variants of~\eqref{eq:Korn3} on bounded domains have been studied in several contributions \cite{BauerNeffPaulyStarke, ContiGarroni,GLP,GmSp,Lewintan2,Lewintan3,Lewintan4,Lewintan5,Lewintan6,NeffPaulyWitsch}.
Inequality~\eqref{eq:Korn3} asserts that the symmetric part of $P$ and the curl of $P$ are strong enough to control the entire matrix field $P$. However, it does not clarify the decisive feature of the symmetric part to make inequalities as~\eqref{eq:Korn3} work and thus has left the question of the sharp underlying mechanisms for~\eqref{eq:Korn3} open.

In this paper, we aim to close this gap. Different from previous contributions, where such inequalities were studied for specific {combinations} $(p,n)$ or particular choices of matrix parts $\mathscr{A}[P]$ such as $\sym P$ or $\dev\sym P$,\footnote{$\dev X\coloneqq X-\frac{\tr X}{n}\cdot\text{\usefont{U}{bbold}{m}{n}1}_{n}$ denotes the deviatoric (trace-free) part of an $n\times n$ matrix $X$.} the purpose of the present paper is to classify those parts $\mathscr{A}[P]$ of matrix fields $P$ such that~\eqref{eq:Korn3} holds with $\sym P$ being replaced by $\mathscr{A}[P]$ for all possible choices of $1\leq p\leq\infty$ depending on the underlying space dimension $n$. We now proceed to give the detailled results and their context each.
\section{Main results}
\subsection{All dimensions estimates}
To approach~\eqref{eq:Korn3} in light of the failure of~\eqref{eq:Korn2}, one may employ a Helmholtz decomposition to represent $P$ as the sum of a gradient (hence curl-free part) and a curl (hence divergence-free part). Under suitable assumptions on $\mathscr{A}$, the gradient part can be treated by a general version of Korn-type inequalities implied by the usual \textsc{Calder\'{o}n-Zygmund} theory~\cite{CZ}. On the other hand, the div-free part is dealt with by the fractional integration theorem (cf.~Lemma~\ref{lem:FIT} below) in the superlinear growth regime $p>1$. If $p=1$, then the particular structure of the div-free term in $n\geq 3$ dimensions allows to reduce to the  \textsc{Bourgain-Brezis} theory \cite{BoBre}. As a starting point, we thus first record the solution of the classification problem {for combinations $(p,n)\neq(1,2)$}:
\begin{theorem}[KMS-inequalities {for $(p,n)\neq(1,2)$}]\label{thm:main1}
Let {$n=2$ and $1<p<2$ or} $n\geq 3$ and $1\leq p<n$. Given a linear map $\mathscr{A}\colon\R^{{m}\times n}\to\R^{{N}}${, with $m, N\in\N$,} the following are equivalent: 
\begin{enumerate}
\item There exists a constant $c=c(p,n,\mathscr{A})>0$ such that the inequality
\begin{align}\label{eq:columbo}
\norm{P}_{\lebe^{p^{*}}(\R^{n})}\leq c\,\Big(\norm{\mathscr{A}[P]}_{\lebe^{p^{*}}(\R^{n})}+\norm{\Curl P}_{\lebe^{p}(\R^{n})}\Big)
\end{align}
holds for all $P\in\hold_{c}^{\infty}(\R^{n};\R^{{m}\times n})$. 
\item $\mathscr{A}$ induces an \emph{elliptic differential operator}, meaning that $\mathbb{A}u\coloneqq \mathscr{A}[\D u]$ is an elliptic differential operator; see Section~\ref{sec:notdiffops} for this terminology. 
\end{enumerate}
\end{theorem}
For the particular choice $\mathscr{A}[P]=\sym P$, this gives us a global variant of \cite[Thm.~2]{ContiGarroni} by \textsc{Conti \& Garroni}. Specifically, Theorem~\ref{thm:main1} is established by generalising the approach of \textsc{Spector} and the first author \cite{GmSp} from $n=3$ to $n\geq 3$ dimensions, and the quick proof is displayed in Section~\ref{sec:sharpAllD}. As a key point, though, we emphasize that ellipticity of $\A$ suffices for~\eqref{eq:columbo} to hold.

As mentioned above, if $n\geq 3$, the borderline estimate for $p=1$ is a consequence of a Helmholtz decomposition and the \textsc{Bourgain-Brezis} estimate 
\begin{align}\label{eq:BB}
\norm{f}_{\lebe^{\frac{n}{n-1}}(\R^{n})}\leq c\norm{\curl f}_{\lebe^{1}(\R^{n})}\qquad\text{for all}\;f\in\hold_{c,\div}^{\infty}(\R^{n};\R^{n}). 
\end{align}
Even if $n=3$, interchanging $\div$ and $\curl$ in~\eqref{eq:BB} is not allowed here, as can be seen by considering regularisations of the gradients of the fundamental solution $\Phi(\cdot)=\frac{1}{3\omega_{3}}|\cdot|^{-1}$ of the Laplacian (and $\rho_{\varepsilon}(\cdot)=\frac{1}{\varepsilon^{3}}\rho(\frac{\cdot}{\varepsilon})$ for a standard mollifier $\rho$): Let $(\varepsilon_{i}),(r_{i})\subset(0,\infty)$ satisfy $\varepsilon_{i}\searrow 0$ and $r_{i}\nearrow\infty$ and choose $\varphi_{r_{i}}\in\hold_{c}^{\infty}(\R^{n})$ with  $\mathbbm{1}_{\ball_{r_{i}}(0)}\leq \varphi_{r_{i}}\leq \mathbbm{1}_{\ball_{2r_{i}}(0)}$ with $|\nabla^{l}\varphi_{r_{i}}|\leq 4/r_{i}^{l}$ for $l\in\{1,2\}$. Putting  
\begin{align}
f_{i}\coloneqq \nabla g_{i}\coloneqq \nabla\Big(\rho_{\varepsilon_{i}}*\Big(\varphi_{r_{i}}\frac{1}{|\cdot|}\Big)\Big),
\end{align}
we then have $\sup_{i\in\N}\norm{\Delta g_{i}}_{\lebe^{1}(\R^{3})}<\infty$, and  validity of the corresponding modified inequality would imply the contradictory estimate $\sup_{i\in\N}\norm{f_{i}}_{\lebe^{\frac{3}{2}}(\R^{3})}=\sup_{i\in\N}\norm{\nabla g_{i}}_{\lebe^{\frac{3}{2}}(\R^{3})}<\infty$. 

However, as can be seen from Example \ref{ex:BBfail} below, inequality~\eqref{eq:BB} fails to extend to $n=2$ dimensions. 
Still, \textsc{Garroni} et al. \cite[\S 5]{GLP} proved validity of a variant of inequality~\eqref{eq:Korn3} in $n=2$ dimensions for the particular choice of the symmetric part of a matrix, and so one might wonder whether Theorem~\ref{thm:main1} remains valid {for the remaining case $(p,n)=(1,2)$} as well. This, however, is \emph{not the case}, as can be seen from 
\begin{example}\label{ex:BBfail}
Consider the trace-free symmetric part of a matrix
\begin{align*}
\dev\sym P=\frac{1}{2}\begin{pmatrix} P_{11}-P_{22} & P_{12}+P_{21} \\ P_{12}+P_{21} & P_{22}-P_{11}\end{pmatrix}\qquad\text{for}\;P=\begin{pmatrix}P_{11} & P_{12} \\ P_{21} & P_{22} \end{pmatrix},
\end{align*}
where $\dev\sym$ induces the usual trace-free symmetric gradient $\varepsilon^{D}(u)=\dev\sym \D u$ which is \emph{elliptic} (cf. Section~\ref{sec:notdiffops}).  Let us now consider for $f\in\hold_{c}^{\infty}(\R^{2})$ the matrix field
\begin{align*}
P_{f}=\begin{pmatrix} \partial_{1}f & \partial_{2}f \\ -\partial_{2}f & \partial_{1}f \end{pmatrix}.
\end{align*}
Then we have 
\begin{align*}
\dev\sym P_{f}=\begin{pmatrix} 0 & 0 \\ 0 & 0 \end{pmatrix}\;\;\;\text{and}\;\;\;\Curl P_{f}=\begin{pmatrix} 0 \\ \Delta f \end{pmatrix},
\end{align*}
If Theorem~\ref{thm:main1} were correct {for the combination $(p,n)=(1,2)$}, the ellipticity of $\varepsilon^{D}$ would  give us 
\begin{align*}
\norm{\nabla f}_{\lebe^{2}(\R^{2})} & = c \norm{P_{f}}_{\lebe^{2}(\R^{2})} \\
&  \stackrel{!}{\leq} c\,\Big(\norm{\dev\sym P_f}_{\lebe^{2}(\R^{2})}+\norm{\Curl P_f}_{\lebe^{1}(\R^{2})} \Big) = c\norm{\Delta f}_{\lebe^{1}(\R^{2})}, 
\end{align*}
and this inequality is again easily seen to be false by taking regularisations and smooth cut-offs of the fundamental solution  of the Laplacian, $\Phi(x)=\frac{1}{2\pi}\log(|x|)$. 
\end{example}

In conclusion, Theorem~\ref{thm:main1} does not extend  to {the case $(p,n)=(1,2)$}. As the strength of $\Curl$ decreases when passing from $n\geq 3$ to $n=2$, one expects that validity of inequalities
\begin{align}\label{eq:KMS2}
\norm{P}_{\lebe^{{1}^{*}}(\R^{2})}\leq c\,\Big(\norm{\mathscr{A}[P]}_{\lebe^{{1}^{*}}(\R^{2})} + \norm{\Curl P}_{\lebe^{{1}}(\R^{2})}\Big),\qquad P\in\hold_{c}^{\infty}(\R^{2};\R^{2\times 2}),
\end{align}
in general requires stronger conditions on $\mathscr{A}$. In this regard, our second main result asserts that the critical {case} $(p,n)=(1,2)$ necessitates very strong algebraic conditions on the matrix parts $\mathscr{A}$ indeed: 
\begin{theorem}[KMS-inequalities in {the case $(p,n)=(1,2)$}]\label{thm:main2}
Let $n=2$ and $p=1$, {then the following are equivalent: 
\begin{enumerate}
\item\label{item:KMScriticalMain1} The critical Korn-Maxwell-Sobolev estimate \eqref{eq:KMS2} holds. 
\item\label{item:KMScriticalMain2} The part map $\mathscr{A}$ induces a \emph{$\mathbb{C}$-elliptic} differential operator. 
\item\label{item:KMScriticalMain3} The part map $\mathscr{A}$ induces an \emph{elliptic and cancelling}  differential operator.
\item\label{item:KMScriticalMain4} The part map $\mathscr{A}$ induces an elliptic differential operator and we have $\dim(\mathscr{A}[\R^{2\times 2}])\in \{3,4\}$. 
\end{enumerate}
Here, we understand by \emph{inducing a differential operator} $\mathbb{A}$ that the associated differential operator $\mathbb{A}u=\mathscr{A}[\D u]$, $u\in\hold_{c}^{\infty}(\R^{2};\R^{2})$, satisfies the corresponding properties. For the precise meaning of $\mathbb{C}$-ellipticity and ellipticity and cancellation, the reader is referred to Section \ref{sec:notdiffops}. 
}
\end{theorem}

Properties \ref{item:KMScriticalMain2} and \ref{item:KMScriticalMain3} express in which sense mere ellipticity has to be strengthened in order to yield the criticial Korn-Maxwell-Sobolev inequality \eqref{eq:KMS2}. Working from Example \ref{ex:BBfail}, condition \ref{item:KMScriticalMain3} is natural from the perspective of limiting $\lebe^{1}$-estimates on the entire space (cf. \textsc{Van Schaftingen} \cite{Va11}). It is then due to the specific dimensional choice $n=2$ that $\mathbb{C}$-ellipticity, usually playing a crucial role in boundary estimates, coincides with ellipticity and cancellation; see \textsc{Rai\c{t}\u{a}} and the first author \cite{GmRa1} and Lemma \ref{lem:helpfulDiffOp} below for more detail. In establishing necessity and sufficiency of these conditions to yield \ref{item:KMScriticalMain1} we will however employ direct consequences of the corresponding features on the induced differential operators, which is why Theorem \ref{thm:main2} appears in the above form. Lastly, the dimensional description of $\mathbb{C}$-ellipticity in the sense of \ref{item:KMScriticalMain4} is a by-product of the proof, which might be of independent interest. 
\begin{figure}
{\begin{tabular}{cc|c|c|}
& & \multicolumn{2}{ c| }{Dimension $n$} \\ \cline{3-4}
&&  $n=2$ & $n\geq 3$ \\ 
\hline
\multicolumn{1}{ c|  }{\multirow{2}{*}{Exponent $p$}} & $p=1$ &  $\mathbb{C}$-ellipticity/ ellipticity and cancellation & ellipticity \\ 
\cline{2-4}\multicolumn{1}{ c|  }{}          
& $p>1$  & ellipticity & ellipticity \\ \hline 
\end{tabular}}
\vspace{0.5cm}
\caption{Sharp conditions on the non-differential part $P\mapsto\mathscr{A}[P]$ to yield the optimal KMS-inequalities in $n$-dimensions.  }\label{fig:outcome}
\end{figure}

Since $\mathscr{A}=\sym$ induces a $\mathbb{C}$-elliptic differential operator in $n=2$ dimensions whereas $\mathscr{A}=\dev\sym$ does not, this result clarifies validity of~\eqref{eq:KMS2} for $\mathscr{A}=\sym$ and its failure for $\mathscr{A}=\dev\sym$. If $(p,n)=(1,2)$, by the non-availability of~\eqref{eq:BB}, inequalities~\eqref{eq:Korn2} cannot be approached by invoking Helmholtz decompositions and using critical estimates on the solenoidal parts. As known from \cite{GLP}, in the particular case of $\mathscr{A}[P]=\sym P$ one may use the specific structure of symmetric matrices (and their complementary parts, the skew-symmetric matrices) to deduce estimate~\eqref{eq:KMS2}. The use of this particularly simple structure also enters proofs of various variants of Korn-Maxwell inequalities. A typical instance is the use of \textsc{Nye}'s formula, allowing to express $\D P$ in terms of $\Curl P$ in three dimensions provided $P$ takes values in the set of skew-symmetric matrices $\mathfrak{so}(3)$; also see the discussion in \cite[Sec.~1.4]{Lewintan2} by \textsc{M\"{u}ller} and the second named authors.

However, in view of classifying the parts $\mathscr{A}[P]$ for which~\eqref{eq:KMS2} holds, we \emph{cannot} utilise the simple structure of  symmetric matrices. To resolve this issue, we will introduce so-called \emph{almost complementary parts} for the sharp class of parts $\mathscr{A}[P]$ for which~\eqref{eq:KMS2} can hold at all, and establish that these almost complementary parts have a sufficiently simple structure to get suitable access to strong Bourgain-Brezis estimates. 

\subsection{Subcritical KMS-inequalities and other variants}
The inequalities considered so far scale and thus the exponent $p^{*}$ cannot be improved for a given $p$. On balls $\ball_{r}(0)$, one might wonder which conditions on $\mathscr{A}$ need to be imposed for inequalities 
\begin{align}\label{eq:KMS3}
\Big(\dashint_{\ball_{r}(0)}|P|^{q}\dif x\Big)^{\frac{1}{q}}\leq c\,\Big(\dashint_{\ball_{r}(0)}|\mathscr{A}[P]|^{q}\dif x\Big)^{\frac{1}{q}} + c\,r\,\Big(\dashint_{\ball_{r}(0)}|\Curl P|^{p}\dif x \Big)^{\frac{1}{p}}
\end{align}
for all $P\in\hold_{c}^{\infty}(\ball_{r}(0);\R^{m\times n})$ 
to hold  with $c=c(p,q,n,\mathscr{A})>0$ provided $q$ is \emph{strictly less than the optimal exponent $p^{*}$}. Since the exponent $q$ here is strictly less than the non-improvable choice $p^{*}$, one might anticipate that even in the case $n=2$ ellipticity of the operator $\mathbb{A}u\coloneqq \mathscr{A}[\D u]$ alone suffices. Indeed we have the following
\begin{theorem}[Subcritical KMS-inequalities]\label{thm:main3}
Let $n\geq 2$, $m, N\in\N$, $1\leq p<\infty$ and $1\leq q<p^{*}$. Then the following hold: 
\begin{enumerate}
\item\label{item:qlessp1} Let $q=1$. Writing $\D u=\sum_{i=1}^{n}\mathbb{E}_{i}\,\partial_{i}u$ and $\A u\coloneqq \mathscr{A}[\D u]=\sum_{i=1}^{n}\mathbb{A}_{i}\,\partial_{i}u$ for suitable $\mathbb{E}_{i},\mathbb{A}_{i}\in\mathscr{L}(\R^{m};\R^{N})$ and all $u\in\hold_{c}^{\infty}(\R^{n};\R^{m})$,~\eqref{eq:KMS3} holds if and only if there exists a linear map $\mathrm{L}\colon\R^{N\times m}\to\R^{N\times m}$ such that $\mathbb{E}_{i}=\mathrm{L}\circ\,\mathbb{A}_{i}$ for all $1\leq i\leq n$. 
\item\label{item:qlessp2} Let $1<q<p^{*}$ if $p<n$ and $1<q<\infty$ otherwise. Then~\eqref{eq:KMS3} holds if and only if $\mathscr{A}$ induces an elliptic differential operator by virtue of $\mathbb{A}u=\mathscr{A}[\D u]$. 
\end{enumerate}
\end{theorem}
Let us remark that Theorems~\ref{thm:main1}--\ref{thm:main3} all deal with exponents $p<n$. Inequalities that address the case $p\geq n$ and hereafter involve different canonical function spaces, are discussed in Sections~\ref{sec:sharpAllD} and~\ref{sec:subcritical}. We conclude this introductory section by discussing several results and generalisations  which appear as special cases of the results displayed in the present paper:
\begin{remark}
\begin{itemize}
\item[(i)] For the particular choice $\mathscr{A}[P]=\sym P$ and $q=p>1$, Theorem~\ref{thm:main3} yields with $n=3$ the result from \cite[Thm.~3]{Lewintan6} and for $n\ge2$ the result from \cite[Thm.~8]{Lewintan5}.
\item[(ii)] For the particular choice of $\mathscr{A}[P]=\sym P$ and $(p,n)=(1,2)$, Theorem~\ref{thm:domains} yields \cite[Thm.~11]{GLP} as a special case. 
\item[(iii)] For $\mathscr{A}[P]=\dev\sym P$ and $q=p>1$, Theorem~\ref{thm:main3} yields \cite[Thm. 3.5]{Lewintan4} and moreover catches up with the missing proof of the trace-free Korn-Maxwell inequality in $n=2$ dimensions.
\item[(iv)] For $\mathscr{A}[P]\in\{\sym P,\dev\sym P\}$, Theorems~\ref{thm:main1}--\ref{thm:main3} yield the sharp, scaling-optimal generalisation of  ~\cite[Thm. 3.5]{Lewintan4}, ~\cite[Thm.~8]{Lewintan5} and  ~\cite[Thm.~3]{Lewintan6}.
\item[(v)] For the part map $\mathscr{A}[P]=\skew P +\tr P\cdot\text{\usefont{U}{bbold}{m}{n}1}_{n}$ as arising e.g. in the study of time-incremental infinitesimal elastoplasticity ~\cite{NeffKnees}, now appears as a special case, cf.~Example~\ref{exampleNeffKnees}.
\end{itemize}
\end{remark}
\begin{figure}
{\begin{tabular}{cc|c|c|}
& & \multicolumn{2}{ c| }{Dimension $n$} \\ \cline{3-4}
&&  $n=2$ & $n\geq 3$ \\ 
\hline
\multicolumn{1}{ c|  }{\multirow{2}{*}{Exponent $p$}} & $p=1$ & \multicolumn{2}{ c|  }{\multirow{2}{*}{ellipticity}}  \\ 
\cline{2-2}\multicolumn{1}{ c|  }{}          
& $p>1$ & \multicolumn{2}{ r| }{\multirow{3}{*}{}} \\ \hline
\end{tabular}}
\caption{Sharp conditions on the non-differential part $P\mapsto\mathscr{A}[P]$ to yield  subcritical KMS-inequalities in $n$-dimensions.}
\end{figure}
\subsection{Organisation of the paper}
Besides the introduction, the paper is organised as follows: In Section~\ref{sec:prelims} we fix notation and gather preliminary results on differential operators and other auxiliary estimates. Based on our above discussion, we first address the most intricate {case} $(p,n)=(1,2)$ and thus establish Theorem~\ref{thm:main2} in Section~\ref{sec:n=2}. Theorems~\ref{thm:main1} and~\ref{thm:main3}, which do not require the use of almost complementary parts, are then established in Sections~\ref{sec:sharpAllD} and~\ref{sec:subcritical}. Finally, the appendix gathers specific Helmholtz decompositions used in the main part of the paper, contextualises the approaches and results displayed in the main part with previous contributions and provides auxiliary material on weighted Lebesgue and Orlicz functions. 

\section{Preliminaries}\label{sec:prelims}
In this preliminary section we fix notation, gather auxiliary notions and results and provide several examples that we shall refer to throughout the main part of the paper. 
\subsection{General notation}\label{sec:notation}
We denote $\omega_{n}\coloneqq \mathscr{L}^{n}(\ball_{1}(0))$ the $n$-dimensional Lebesgue measure of the unit ball. For $m\in\N_{0}$ and a finite dimensional (real) vector space $X$, we denote the space of $X$-valued polynomials on $\R^{n}$ of degree at most $m$ by $\mathscr{P}_{m}(\R^{n};X)$; moreover, for $1<p<\infty$, we denote ${\dot{\sobo}}{^{1,p}}(\R^{n};X)$ the homogeneous Sobolev space (i.e., the closure of $\hold_{c}^{\infty}(\R^{n};X)$ for the gradient norm $\norm{\D u}_{\lebe^{p}(\R^{n})}$) and ${\dot{\sobo}}{^{-1,p'}}(\R^{n};X):={\dot{\sobo}}{^{1,p}}(\R^{n};X)'$ its dual. Finally, to define the matrix curl $\Curl P$ for $P\colon\R^{n}\to\R^{n\times n}$, we recall from \cite{Lewintan1, Lewintan4} the generalised cross product  $\ntimes{n}:\R^n \times \R^n\to\R^{\frac{n(n-1)}{2}}$. Letting
\begin{align*}
 a =(\overline{a},a_n)^\top\in\R^n\quad \text{and}\quad b =(\overline{b},b_n)^\top\in\R^n \quad \text{with $\overline{a}, \overline{b}\in\R^{n-1}$},
\end{align*}
we inductively define
\begin{align}\label{eq:iductivecross}
 a\ntimes{n}b \coloneqq \begin{pmatrix} \overline{a}\ntimes{n-1}\overline{b} \\[1ex]
                b_n\cdot\overline{a}-a_n\cdot\overline{b}
                 \end{pmatrix}\in\R^{\frac{n(n-1)}{2}}
                \quad \text{where}\quad
\begin{pmatrix}a_1\\a_2  \end{pmatrix} \ntimes{2}
\begin{pmatrix}b_1\\b_2  \end{pmatrix} \coloneqq a_1\,b_2-a_2\,b_1.
\end{align}
The generalised cross product $a\ntimes{n}\cdot$ is linear in the second component and thus can be identified with a multiplication with a matrix denoted by $\nMat{a}{n}\in\R^{\frac{n(n-1)}{2}\times n}$ so that
\begin{align}\label{eq:frobenius}
 a\ntimes{n}b\eqqcolon \nMat{a}{n}b\qquad \text{for all } \ b\in\R^n.
\end{align}
For a vector field $a\colon\R^{n}\to\R^{n}$ and matrix field $P\colon\R^{n}\to\R^{m\times n}$, with $m\in\N$, we finally declare $\curl a$ and $\Curl P$ via
\begin{align}\label{eq:curldef}
\begin{split}
\curl a\coloneqq a\ntimes{n}(-\nabla)=\nMat{\nabla}{n}a,\qquad
\Curl P\coloneqq P\ntimes{n}(-\nabla)= P\nMat{\nabla}{n}^\top.
\end{split}
\end{align}
\subsection{Notions for differential operators}\label{sec:notdiffops}
Let $\mathbb{A}$ be a first order, linear, constant coefficient differential operator on $\R^{n}$ between two finite-dimensional real inner product spaces $V$ and $W$. In consequence, there exist linear maps $\A_{i}\colon V\to W$, $i\in\{1,...,n\}$ such that
\begin{align}\label{eq:form}
\A u = \sum_{i=1}^{n}\A_{i}\partial_{i}u.
\end{align}
For $\Omega\subseteq\R^n$ open and $1\le q\le \infty$, we set 
\begin{align}\label{eq:Waqspace}
\sobo^{\A,q}(\Omega)\coloneqq\{u\in\lebe^{q}(\Omega;V)\mid \norm{u}_{\lebe^q(\Omega)}+\norm{\A u}_{\lebe^{q}(\Omega)}<\infty \}.
\end{align}
With an operator of the form~\eqref{eq:form}, we associate the \emph{symbol map}
\begin{align}\label{eq:fouriersymbol}
\A[\xi]\colon V\to W,\qquad\A[\xi]v\coloneqq \sum_{i=1}^{n}\xi_{i}\A_{i}v,\qquad \xi\in\R^{n},v\in V. 
\end{align}
Following \cite{BDG}, we denote 
\begin{align}\label{eq:bilinearpairing}
\otimes_{\A}\colon V\times\R^{n}\to W,\qquad\otimes_{\A}(\xi,v)\coloneqq v\otimes_{\A}\xi \coloneqq  \A[\xi]v. 
\end{align}
In the specific case where $V=\R^{m}$, $W=\R^{m\times n}$ and $\A=\D$ \ is the derivative, this gives us back the usual tensor product $v\otimes\xi=v\, \xi^{\top}$; if $V=\R^{n}$, $W=\operatorname{Sym}(n)=\R_{\sym}^{n\times n}$, it gives us back the usual symmetric tensor product $v\odot\xi=\frac{1}{2}(v\otimes\xi+\xi\otimes v)$. 

Elements $w\in W$ of the form $w=\Atimes{v}{\xi}$ are called \emph{pure $\A$-tensors}. We now define 
\begin{align*}
\mathscr{R}(\A)\coloneqq \mathrm{span}(\{\Atimes{a}{b}\mid a\in V,\;b\in\R^{n}\}). 
\end{align*}
Moreover, if $\{\mathbf{a}_{1},...,\mathbf{a}_{N}\}$ is a basis of $\R^{N}$ and $\{\mathbf{e}_{1},...,\mathbf{e}_{n}\}$ is a basis of $\R^{n}$, then by linearity the set 
$\{\mathbf{a}_{j}\otimes_{\A}\mathbf{e}_{i}\mid 1\leq i\leq n,\;1\leq j \leq N\}$ spans $\mathscr{R}(\A)$ and hence contains a basis of $\mathscr{R}(\A)$.  For our future objectives, it is important to note that this result holds irrespectively of the particular choice of bases. 

We now recall some notions on differential operators gathered from \cite{BDG,Hoermander,Smith,Spencer,Va11}. An operator of the form~\eqref{eq:form} is called \emph{elliptic} provided the Fourier symbol $\A[\xi]\colon V\to W$ is injective for any $\xi\in\R^{n}\backslash\{0\}$. This is equivalent to the Legendre-Hadamard ellipticity of $W(P)=\frac12|\mathscr{A}[P]|^2$ in the sense that the bilinear form $\D^2_P W(P)(\xi\otimes\eta,\xi\otimes\eta)=|\mathscr{A}(\xi\otimes \eta)|^2>0$ is strictly positive definite. A strengthening of this notion is that of \emph{$\mathbb{C}$-ellipticity}. Here we require that for any $\xi\in\mathbb{C}^{n}\backslash\{0\}$ the associated symbol map $\A[\xi]\colon V+\imag V\to W+\imag W$ is injective. 

\begin{lemma}[\cite{GmRa1,GmRaVa1,Va11}]\label{lem:helpfulDiffOp}
Let $\A$ be a first order differential operator of the form~\eqref{eq:form} with $n=2$. Then the following are equivalent: 
\begin{enumerate}
\item\label{item:CA1} $\A$ is $\mathbb{C}$-elliptic. 
\item\label{item:CA2} $\A$ is elliptic and cancelling, meaning that $\A$ is elliptic in the above sense and 
\begin{align*}
\bigcap_{\xi\in\R^{2}\backslash\{0\}}\A[\xi](V)=\{0\}. 
\end{align*}
\item\label{item:CA3} There exists a constant $c>0$ such that the Sobolev estimate $\norm{u}_{\lebe^{2}(\R^{2})}\leq c\norm{\A u}_{\lebe^{1}(\R^{2})}$ holds for all $u\in\hold_{c}^{\infty}(\R^{2};V)$. 
\end{enumerate}
\end{lemma}
This equivalence does not hold true in $n\geq 3$ dimensions (see~\cite[Sec.~3]{GmRaVa1} for a discussion; for general dimensions $n\geq 2$, one has $\text{\ref{item:CA2}}\Leftrightarrow\text{\ref{item:CA3}}$, see~\cite{Va11}, but only $\text{\ref{item:CA1}}\Rightarrow\text{\ref{item:CA2}}$, see~\cite{GmRa1,GmRaVa1}). The following lemma is essentially due to \textsc{Smith}~\cite{Smith}; also see \cite{DieningGmeineder,Kalamajska} for the explicit forms of the underlying Poincar\'{e}-type inequalities:
\begin{lemma}\label{lem:Poincare} 
An operator $\mathbb{A}$ of the form~\eqref{eq:form} is $\mathbb{C}$-elliptic if and only if there exists $m\in\mathbb{N}_{0}$ such that $\ker(\A)\subset\mathscr{P}_{m}(\R^{n};V)$. Moreover, for any open, bounded and connected Lipschitz domain $\Omega\subset\R^{n}$ there exists $c=c({q},n,\A,\Omega)>0$ such that for any $u\in\sobo^{\A,{q}}(\Omega)$ there is $\Pi_{\A}u\in\ker(\A)$ such that 
\begin{align*}
\norm{u-\Pi_{\A}u}_{\lebe^{{q}}(\Omega)}\leq c\norm{\mathbb{A}u}_{\lebe^{{q}}(\Omega)}. 
\end{align*}
\end{lemma}

We conclude with the following example to be referred to frequently, putting the classical operators $\nabla$, $\varepsilon$ and $\varepsilon^{D}$ into the framework of~\eqref{eq:form}. For this, it is convenient to make the identification 
\begin{align*}
\begin{pmatrix} a & b \\ c & d \end{pmatrix}\longleftrightarrow \begin{pmatrix} a \\ b \\ c \\ d \end{pmatrix},\qquad a,b,c,d\in\R.
\end{align*}
\begin{example}[Gradient, deviatoric gradient, symmetric gradient and deviatoric symmetric gradient]
In the following, let $n=2$, $V=\R^{2}$ and $W=\R^{2\times 2}$. The derivative fits into the framework~\eqref{eq:form} by taking 
\begin{align}\label{eq:gradsymb}\tag{$\nabla$}
\A_{1} = \begin{pmatrix} 1 & 0 \\ 0 & 0 \\ 0 & 1 \\ 0 & 0 \end{pmatrix},\;\;\;\A_{2} = \begin{pmatrix} 0 & 0 \\ 1 & 0 \\ 0 & 0 \\ 0 & 1 \end{pmatrix}, 
\end{align}
the deviatoric gradient $\nabla^{D}u\coloneqq \dev\D u =\D u-\frac{1}{n}\div u\cdot\text{\usefont{U}{bbold}{m}{n}1}_{n}$ by taking 
\begin{align}\label{eq:devgradsymb}\tag{$\nabla^{D}$}
\A_{1} = \begin{pmatrix} \frac{1}{2} & 0 \\ 0 & 0 \\ 0 & 1 \\ -\frac{1}{2} & 0 \end{pmatrix},\;\;\;\A_{2} = \begin{pmatrix} 0 & -\frac{1}{2} \\ 1 & 0 \\ 0 & 0 \\ 0 & \frac{1}{2} \end{pmatrix}, 
\end{align}
whereas the symmetric gradient is recovered by taking
\begin{align}\label{eq:symgradsymb}\tag{$\varepsilon$}
\A_{1} = \begin{pmatrix} 1 & 0 \\ 0 & \frac{1}{2} \\ 0 & \frac{1}{2} \\ 0 & 0 \end{pmatrix},\;\;\;\A_{2} = \begin{pmatrix} 0 & 0 \\ \frac{1}{2} & 0 \\ \frac{1}{2} & 0 \\ 0 & 1 \end{pmatrix}.  
\end{align}
Finally, the deviatoric symmetric gradient $\varepsilon^{D}(u)\coloneqq \dev\sym\D u=\sym \D u -\frac{1}{n}\div u\cdot\text{\usefont{U}{bbold}{m}{n}1}_{n}$ is retrieved by 
\begin{align}\label{eq:devsymgradsymb}\tag{$\varepsilon^{D}$}
\A_{1} = \begin{pmatrix} \frac{1}{2} & 0 \\ 0 & \frac{1}{2} \\ 0 & \frac{1}{2} \\ -\frac{1}{2} & 0 \end{pmatrix},\;\;\;\A_{2} = \begin{pmatrix} 0 & -\frac{1}{2} \\ \frac{1}{2} & 0 \\ \frac{1}{2} & 0 \\ 0 & \frac{1}{2} \end{pmatrix}.
\end{align}
Similar representations can be found in higher dimensions. However, let us remark that whereas $\nabla,\nabla^{D}$ and $\varepsilon$ are $\mathbb{C}$-elliptic in all dimensions $n\geq 2$, the trace-free symmetric gradient is $\mathbb{C}$-elliptic precisely in $n\geq 3$ dimensions (cf.~\cite[Ex.~2.2]{BDG}). Another class of operators that arises in the context of infinitesimal elastoplasticity but is handled most conveniently by the results displayed below, is discussed in Example~\ref{exampleNeffKnees}.  
\end{example}

\subsection{{Miscellaneous} bounds} 
In this section, we record some auxiliary material on integral operators. Given $0<s<n$, the \emph{$s$-th order Riesz potential} $\mathcal{I}_{s}(f)$ of a locally integrable function $f$ is defined by 
\begin{align}
\mathcal{I}_{s}(f)(x)\coloneqq c_{s,n}\int_{\R^{n}}\frac{f(y)}{|x-y|^{n-s}}\dif y,\qquad x\in\R^{n}, 
\end{align}
where $c_{s,n}>0$ is a suitable finite constant, the precise value of which shall not be required in the sequel. We now have 
\begin{lemma}[Fractional integration, {\cite[Thm.~3.1.4]{AdamsHedberg},\,\cite[\S 2.7]{diBenedetto},\cite[Sec. V.1]{Stein}}]\label{lem:FIT}
Let $n\geq 2$ and $0<s<n$. 
\begin{enumerate}
\item\label{item:FItcrit} For any $1<p<\infty$ with $s\,p<n$ the Riesz potential $\mathcal{I}_{s}$ is a bounded linear operator $\mathcal{I}_{s}\colon\lebe^{p}(\R^{n})\to \lebe^{\frac{np}{n-sp}}(\R^{n})$. 
\item\label{item:FItsubc} Let $1\leq p<\infty$. For any $1\leq q <\frac{np}{n-sp}$ if $s\,p<n$ or all $1\leq q<\infty$ if $sp\geq n$ and any $r>0$ there exists $c=c(p,q,n,s,r)>0$ such that we have 
\begin{align*}
\norm{\mathcal{I}_{s}(\mathbbm{1}_{\ball_{r}(0)}f)}_{\lebe^{q}(\ball_{r}(0))}\leq c\norm{f}_{\lebe^{p}(\ball_{r}(0))}\qquad\text{for all}\;f\in\hold^{\infty}(\R^{n}).
\end{align*}
\end{enumerate}
\end{lemma}
In the regime $s=0$ we require two other ingredients as follows. The first one is  
\begin{lemma}[{Ne\v{c}as-Lions}, {\cite[Thm.~1]{Necas}}]\label{lem:Necas}
Let $\Omega\subset\R^{n}$ be open and bounded with Lipschitz boundary and let $N\in\mathbb{N}$. Then for any $1<p<\infty$ there exists a constant $c=c(p,n,N,\Omega)>0$ such that we have
\begin{align*}
\norm{f}_{\lebe^{p}(\Omega)}\leq c\,\Big(\norm{f}_{\sobo^{-1,p}(\Omega)}+\norm{\D f}_{\sobo^{-1,p}(\Omega)} \Big)\qquad\text{for all}\;f\in\mathscr{D}'(\Omega;\R^{N}). 
\end{align*}
\end{lemma} 
Whereas the preceding lemma proves particularly useful in the context of bounded domains (also see Theorem~\ref{thm:domains} below), the situation on the entire $\R^{n}$ can be accessed by Calder\'{o}n-Zygmund estimates~\cite{CZ} (see~\cite[Prop.~4.1]{Va11} or~\cite[Prop.~4.1]{ContiGmeineder} for related arguments) and implies
\begin{lemma}[Calder\'{o}n-Zygmund-Korn]\label{lem:CZK}
Let $\A$ be a differential operator of the form~\eqref{eq:form} and let $1<p<\infty$. Then $\A$ is elliptic if and only if there exists a constant $c=c(p,n,\A)>0$ such that we have
\begin{align*}
\norm{\D u}_{\lebe^{p}(\R^{n})}\leq c\norm{\A u}_{\lebe^{p}(\R^{n})}\qquad\text{for all}\;u\in\hold_{c}^{\infty}(\R^{n};V).
\end{align*}
\end{lemma}
\section{Sharp KMS-inequalities: The case $(p,n)=(1,2)$}\label{sec:n=2}
\subsection{Almost complementary parts}
Throughout this section, let $n=2$. The \emph{complementary part} of $P$ with respect to $\mathscr{A}\colon\R^{2\times 2}\to\R^{2\times 2}$ is given by $P-\mathscr{A}[P]$. In the present section we establish that, if $\mathscr{A}$ induces a $\mathbb{C}$-elliptic differential operator, then for a suitable linear map $\mathcal{L}\colon\R^{2\times 2}\to\R^{2\times 2}$ the image $(\mathrm{Id}-\mathcal{L}\mathscr{A}[\cdot])(\R^{2\times 2})$ is one-dimensional, in fact a line spanned by some invertible matrix. We then shall refer to $P-\mathcal{L}\mathscr{A}[P]$ as the \emph{almost complementary part}. 

To this end, note that if $\A u=\mathscr{A}[\D u]$ for all $u\in\hold_{c}^{\infty}(\R^{2};\R^{2})$ with some linear map $\mathscr{A}\colon\R^{2\times 2}\to\R^{2\times 2}$, then the algebraic relation 
\begin{align}\label{eq:eckartwitzigmann}
\mathscr{A}[\mathbf{e}\otimes\mathbf{f}]=\mathbf{e}\otimes_{\A}\mathbf{f}=\A[\mathbf{f}]\mathbf{e}\qquad\text{for all}\;\mathbf{e},\mathbf{f}\in\R^{2}
\end{align}
holds. Within this framework, we now have
\begin{proposition}\label{prop:dodgethedodo}
Suppose that $\A$ is $\mathbb{C}$-elliptic. Then there exists a linear map $\mathcal{L}\colon\R^{2\times 2}\to\R^{2\times 2}$ and some $\Q\in\mathrm{GL}(2)$ such that 
\begin{align}\label{eq:gordonramsey}
\{X-\mathcal{L}(\mathscr{A}[X])\mid X\in\R^{2\times 2}\}=\R \Q.
\end{align} 
\end{proposition}
\begin{proof}
Let $\{\mathbf{e}_{1},\mathbf{e}_{2}\}$ be a basis of $\R^{2}$. As $\{\mathbf{e}_{i}\otimes\mathbf{e}_{j}\mid (i,j)\in\{1,2\}\times\{1,2\}\}$ spans $\R^{2\times 2}$, it suffices to show that there exist a linear map $\mathcal{L}\colon\R^{2\times 2}\to\R^{2\times 2}$, $\Q\in\mathrm{GL}(2)$ and numbers $\gamma_{ij}\in\R$ not all equal to zero such that 
\begin{align}\label{eq:frankrosin}
(\mathbf{e}_{i}\otimes\mathbf{e}_{j})-\mathcal{L}(\mathbf{e}_{i}\otimes_{\A}\mathbf{e}_{j})=\gamma_{ij}\Q\qquad\text{for all}\;i,j\in\{1,2\}. 
\end{align}
If we can establish~\eqref{eq:frankrosin}, we express $X\in\R^{2\times 2}$ as 
\begin{align*}
X=\sum_{1\leq i,j\leq 2}x_{ij}\mathbf{e}_{i}\otimes\mathbf{e}_{j} 
\end{align*}
and use~\eqref{eq:frankrosin} in conjunction with~\eqref{eq:eckartwitzigmann} to find 
\begin{align}
X-\mathcal{L}(\mathscr{A}[X]) & = \sum_{1\leq i,j\leq 2}x_{ij}\big((\mathbf{e}_{i}\otimes\mathbf{e}_{j})-\mathcal{L}(\mathbf{e}_{i}\otimes_{\A}\mathbf{e}_{j})\big) = \Big(\sum_{1\leq i,j\leq 2}\gamma_{ij}x_{ij}\Big)\Q. 
\end{align}
Since not all $\gamma_{ij}$'s equal zero, this gives~\eqref{eq:gordonramsey}. 

Let us therefore establish~\eqref{eq:frankrosin}. We approach~\eqref{eq:frankrosin} by distinguishing several dimensional {cases}. First note that $\mathscr{A}[\R^{2\times 2}]$ is at least two dimensional. In fact, 
\begin{align}\label{eq:rosinsrestaurant}
 \mathbb{A}[\mathbf{e}_{1}](\R^{2})\subset \mathscr{A}[\R^{2\times 2}], 
\end{align}
and since $\mathbb{A}$ is elliptic, $\ker(\A[\mathbf{e}_{1}])=\{0\}$, whence $\dim(\mathbb{A}[\mathbf{e}_{1}](\R^{2}))=2$ by the rank-nullity theorem. Now, if $\dim(\mathscr{A}[\R^{2\times 2}])=2$, we conclude from~\eqref{eq:rosinsrestaurant} that $\mathscr{A}[\R^{2\times 2}]=\A[\mathbf{e}_{1}](\R^{2})$. We conclude that for all $\xi\in\R^{2}\backslash\{0\}$ we have $\A[\xi](\R^{2})=\A[\mathbf{e}_{1}](\R^{2})$ and hence $\A$ fails to be cancelling. Since ellipticity together with cancellation is equivalent to $\mathbb{C}$-ellipticity for first order operators in $n=2$ dimensions by Lemma~\ref{lem:helpfulDiffOp}, we infer that $\dim(\mathscr{A}[\R^{2\times 2}])\in\{3,4\}$. 

If $\dim(\mathscr{A}[\R^{2\times 2}])=4$, we note that $\{\mathbf{e}_{i}\otimes_{\mathbb{A}}\mathbf{e}_{j}\mid 1\leq i,j\leq 2\}$ spans $\mathscr{A}[\R^{2\times 2}]$. Consequently, the vectors $\mathbf{e}_{i}\otimes_{\A}\mathbf{e}_{j}$, $1\leq i,j\leq 2$, form a basis of $\mathscr{A}[\R^{2\times 2}]$. For any given $\Q\in\mathrm{GL}(2)$ we may simply declare $\mathcal{L}$ by its action on these basis vectors via 
\begin{align}\label{eq:michelin}
\mathcal{L}(\mathbf{e}_{i}\otimes_{\A}\mathbf{e}_{j})\coloneqq (\mathbf{e}_{i}\otimes\mathbf{e}_{j})-\Q. 
\end{align}
Then~\eqref{eq:frankrosin} is fulfilled with $\gamma_{ij}=1$ for all $(i,j)\in\{1,2\}\times\{1,2\}$, and we conclude in this case. 

It remains to discuss the case $\dim(\mathscr{A}[\R^{2\times 2}])=3$. Then there exists $(i_{0},j_{0})\in\{1,2\}\times\{1,2\}$ such that 
\begin{align}\label{eq:idea}
\mathbf{e}_{i_{0}}\otimes_{\A}\mathbf{e}_{j_{0}} = \sum_{(i,j)\neq(i_{0},j_{0})}a_{ij}\mathbf{e}_{i}\otimes_{\A}\mathbf{e}_{j} 
\end{align}
for some $a_{ij}\in\R$ not all equal to zero, $(i,j)\neq (i_{0},j_{0})$. In what follows, we assume that $(i_{0},j_{0})=(2,1)$; for the other index {combinations}, the argument is analogous. Our objective is to show that 
\begin{align}\label{eq:defQcritical}
\Q\coloneqq \Big(\begin{matrix} -a_{11} & -a_{12} \\ 1 & -a_{22}\end{matrix}\Big)
\end{align}
is invertible. By assumption, not all coefficients $a_{11},a_{12},a_{22}$ vanish. We distinguish four options: 
\begin{itemize}
\item[(i)] If $a_{22}=0$, $a_{12}\neq 0$, then the matrix $\Q$ from~\eqref{eq:defQcritical} is invertible. 
\item[(ii)] If $a_{22}\neq 0$, $a_{12}=0$, then necessarily $a_{11}\neq 0$ and so $\Q$ is invertible. Indeed, if $a_{11}=0$, then~\eqref{eq:eckartwitzigmann} and~\eqref{eq:idea} yield $\mathbb{A}[\mathbf{e}_{1}]\mathbf{e}_{2}=a_{22}\mathbb{A}[\mathbf{e}_{2}]\mathbf{e}_{2}$. Since $\A$ is of first order and homogeneous, the map $\xi\mapsto\A[\xi]v$ is linear for any fixed $v$. We conclude that $\mathbb{A}[\mathbf{e}_{1}-a_{22}\mathbf{e}_{2}]\mathbf{e}_{2}=0$, and since $\mathbf{e}_{1}-a_{22}\mathbf{e}_{2}\neq 0$ independently of $a_{22}$, this is impossible by ellipticity of $\A$. 
\item[(iii)] If $a_{12}=a_{22}=0$, then~\eqref{eq:eckartwitzigmann} and ~\eqref{eq:idea} yield $\A[\mathbf{e}_{1}]\mathbf{e}_{2}=a_{11}\A[\mathbf{e}_{1}]\mathbf{e}_{1}$. We then conclude that $\mathbb{A}[\mathbf{e}_{1}](\mathbf{e}_{2}-a_{11}e_{1})=0$. But $\mathbf{e}_{2}-a_{11}\mathbf{e}_{1}\neq 0$ independently of $a_{11}$, this is impossible by ellipticity of $\A$. 
\item[(iv)] If $a_{22},a_{12}\neq 0$, then non-invertibility of the matrix is equivalent to a constant $\lambda\neq 0$ such that 
\begin{align}\label{eq:dax}
-a_{11}=-\lambda a_{12},\;\;\;1=-\lambda a_{22}. 
\end{align}
Again using linearity of $\xi\mapsto \A[\xi]v$ for any fixed $v$, ~\eqref{eq:idea} becomes 
\begin{align*}
-\lambda a_{22}\A[\mathbf{e}_{1}]\mathbf{e}_{2} & = a_{11}\mathbb{A}[\mathbf{e}_{1}]\mathbf{e}_{1} + a_{12}\mathbb{A}[\mathbf{e}_{2}]\mathbf{e}_{1} + a_{22}\mathbb{A}[\mathbf{e}_{2}]\mathbf{e}_{2} \\ 
& \!\!\!\stackrel{\eqref{eq:dax}}{=} \lambda a_{12}\mathbb{A}[\mathbf{e}_{1}]\mathbf{e}_{1} + a_{12}\mathbb{A}[\mathbf{e}_{2}]\mathbf{e}_{1} + a_{22}\mathbb{A}[\mathbf{e}_{2}]\mathbf{e}_{2}\\ 
& = \lambda a_{12}\mathbb{A}[\mathbf{e}_{1}]\mathbf{e}_{1} + \mathbb{A}[\mathbf{e}_{2}](a_{12}\mathbf{e}_{1}+a_{22}\mathbf{e}_{2}), 
\end{align*}
whereby 
\begin{align*}
\mathbb{A}[-\lambda\mathbf{e}_{1}](a_{12}\mathbf{e}_{1}+a_{22}\mathbf{e}_{2})=\mathbb{A}[\mathbf{e}_{2}](a_{12}\mathbf{e}_{1}+a_{22}\mathbf{e}_{2}), 
\end{align*}
so that
\begin{align}\label{eq:conan}
\mathbb{A}[\lambda\mathbf{e}_{1}+\mathbf{e}_{2}](a_{12}\mathbf{e}_{1}+a_{22}\mathbf{e}_{2})=0.
\end{align}
By assumption, $a_{12},a_{22}\neq 0$ and hence $a_{12}\mathbf{e}_{1}+a_{22}\mathbf{e}_{2}\neq 0$. But $\lambda\mathbf{e}_{1}+\mathbf{e}_{2}\neq 0$ independently of $\lambda$, and hence~\eqref{eq:conan} is at variance with the  ellipticity of $\A$. 
\end{itemize}  
In conclusion, $\Q$ defined by~\eqref{eq:defQcritical} satisfies $\Q\in\mathrm{GL}(2)$ and moreover 
\begin{align}\label{eq:kogoro}
\Q = \mathbf{e}_{i_{0}}\otimes\mathbf{e}_{j_{0}} - \sum_{(i,j)\neq (i_{0},j_{0})} a_{ij}\mathbf{e}_{i}\otimes\mathbf{e}_{j}, 
\end{align}
where we recall that $(i_{0},j_{0})=(2,1)$. By assumption and because of~\eqref{eq:idea}, $\mathbf{e}_{1}\otimes_{\A}\mathbf{e}_{1},\mathbf{e}_{1}\otimes_{\A}\mathbf{e}_{2}$ and $\mathbf{e}_{2}\otimes_{\A}\mathbf{e}_{2}$ form a basis of $\mathscr{A}[\R^{2\times 2}]$, and we may extend these three basis vectors by a vector $\mathbf{f}\in\R^{2\times 2}$ to a basis of $\R^{2\times 2}$. We then declare the linear map $\mathcal{L}$ by its action on these basis vectors via 
\begin{align}\label{eq:plachutta}
\begin{split}
&\mathcal{L}(\mathbf{e}_{i}\otimes_{\A}\mathbf{e}_{j})\coloneqq \mathbf{e}_{i}\otimes\mathbf{e}_{j}\;\;\;\;\;\;\;\;\text{if}\;(i,j)\neq (i_{0},j_{0}),\\
&\mathcal{L}(\mathbf{f})\coloneqq \mathbf{0}\in\R^{2\times 2}. 
\end{split}
\end{align}
Combining our above findings, we obtain 
\begin{align*}
\mathcal{L}(\mathbf{e}_{i_{0}}\otimes_{\A}\mathbf{e}_{j_{0}}) & \stackrel{\eqref{eq:idea}}{=} \sum_{(i,j)\neq(i_{0},j_{0})}a_{ij}\mathcal{L}(\mathbf{e}_{i}\otimes_{\A}\mathbf{e}_{j}) \stackrel{\eqref{eq:plachutta}_{1}}{=} \sum_{(i,j)\neq(i_{0},j_{0})}a_{ij}\mathbf{e}_{i}\otimes\mathbf{e}_{j} \\
& \;= \mathbf{e}_{i_{0}}\otimes\mathbf{e}_{j_{0}} - \Big( \mathbf{e}_{i_{0}}\otimes\mathbf{e}_{j_{0}}-\sum_{(i,j)\neq(i_{0},j_{0})}a_{ij}\mathbf{e}_{i}\otimes\mathbf{e}_{j} \Big)\\ 
&\!\!\stackrel{\eqref{eq:kogoro}}{=} \mathbf{e}_{i_{0}}\otimes\mathbf{e}_{j_{0}} - \Q, 
\end{align*}
and so, in light of~\eqref{eq:plachutta}, we may choose 
\begin{align*}
\gamma_{ij}=\begin{cases} 0&\;\text{if}\;(i,j)\neq (i_{0},j_{0}),\\
1&\;\text{if}\;(i,j)=(i_{0},j_{0})
\end{cases} 
\end{align*}
to see that~\eqref{eq:frankrosin} is fulfilled; in particular, not all $\gamma_{ij}$'s vanish, and $\Q\in\mathrm{GL}(2)$. The proof is complete. 
\end{proof}
The next lemma shows that in two dimensions the operator $\dev\sym$ is indeed a typical example of an elliptic but not $\mathbb{C}$-elliptic operator:
\begin{lemma}[Description of first order $\mathbb{C}$-elliptic operators in two dimensions]\label{lem:characterC} Suppose that $\mathbb{A}$ is elliptic. Then the operator $\mathbb{A}$ given by $\A u:=\mathscr{A}[\D u]$ for $u\colon\R^{2}\to\R^{2}$ is $\mathbb{C}$-elliptic if and only if $\dim (\mathscr{A}[\R^{2\times2}]) \in \{3,4\}$. 
\end{lemma}
\begin{proof}
 The sufficiency part is already contained in the proof of the previous Proposition \ref{prop:dodgethedodo}. For the necessity part consider for $a,b\in\mathbb{C}^2$ with $a=(a_1,a_2)^\top$, $b=(b_1,b_2)^\top$:
\begin{align*}
 \Atimes{a}{b}=&\ \Re(a_1b_1)\Atimes{\eI{1}}{\eI{1}}+\Re(a_1b_2)\Atimes{\eI{1}}{\eI{2}}\\ &+\Re(a_2b_1)\Atimes{\eI{2}}{\eI{1}}+\Re(a_2b_2)\Atimes{\eI{2}}{\eI{2}} \\
 &+\imag\cdot\big[\Im(a_1b_1)\Atimes{\eI{1}}{\eI{1}}+\Im(a_1b_2)\Atimes{\eI{1}}{\eI{2}}\\
 &\qquad +\Im(a_2b_1)\Atimes{\eI{2}}{\eI{1}}+\Im(a_2b_2)\Atimes{\eI{2}}{\eI{2}}\big].
\end{align*}
If $\dim (\mathscr{A}[\R^{2\times2}])=4$, then all $\Atimes{\eI{i}}{\eI{j}}$ are linearly independent over $\R$, and the condition $\Atimes{a}{b}=0$ implies
\begin{align*}
 \Re(a_1b_1)=\Re(a_1b_2) =\Re(a_2b_1)=\Re(a_2b_2)=0 \\
 \shortintertext{and}
 \Im(a_1b_1)=\Im(a_1b_2) =\Im(a_2b_1)=\Im(a_2b_2)=0.
\end{align*}
Thus, with $a_i,b_i\in\mathbb{C}$ we have
\begin{align*}
 a_1b_1=a_1b_2=a_2b_1=a_2b_2=0.
\end{align*}
And for $b\neq0$ we deduce $a=0$, meaning that the operator $\mathbb{A}$ is $\mathbb{C}$-elliptic in this case.\bigskip

In the case $\dim (\mathscr{A}[\R^{2\times2}])=3$ we assume without loss of generality, that there exist $\alpha,\beta,\gamma\in\R$ such that
\begin{align}
 \Atimes{\eI{2}}{\eI{2}}=\alpha\Atimes{\eI{1}}{\eI{1}}+\beta\Atimes{\eI{1}}{\eI{2}}+\gamma\Atimes{\eI{2}}{\eI{1}}.
\end{align}
It follows from the ellipticity assumption on $\mathbb{A}$ (cf. the invertibility of $\Q$ from \eqref{eq:defQcritical} in the proof of Proposition \ref{prop:dodgethedodo}) that these coefficients satisfy
\begin{equation}
 \alpha(-1)-\beta\gamma\neq0  \quad \Leftrightarrow\quad \alpha+\beta\gamma\neq0.
\end{equation}
The condition $\Atimes{a}{b}=0$ with $a,b\in\mathbb{C}$ then implies
\begin{align*}
 0 &=\Re(a_1b_1)+\alpha\,\Re(a_2b_2) \overset{\alpha\in\R}{=}\Re(a_1b_1+\alpha\, a_2b_2),\\
 0 &=\Re(a_1b_2)+\beta\,\Re(a_2b_2) \overset{\beta\in\R}{=}\Re((a_1+\beta\, a_2)b_2),\\
 0 &=\Re(a_2b_1)+\gamma\,\Re(a_2b_2) \overset{\gamma\in\R}{=}\Re(a_2(b_1+\gamma\, b_2)),\\
 0 &=\Im(a_1b_1)+\alpha\,\Im(a_2b_2) \overset{\alpha\in\R}{=}\Im(a_1b_1+\alpha\, a_2b_2),\\
 0 &=\Im(a_1b_2)+\beta\,\Im(a_2b_2) \overset{\beta\in\R}{=}\Im((a_1+\beta\, a_2)b_2),\\
 0 &=\Im(a_2b_1)+\gamma\,\Im(a_2b_2) \overset{\gamma\in\R}{=}\Im(a_2(b_1+\gamma\, b_2)).
\end{align*}
Hence, $a_i,b_i\in\mathbb{C}$ satisfy
\begin{align}\label{eq:cases}
 \begin{cases}
  a_1b_1+\alpha\, a_2b_2  &= 0,\\
  (a_1+\beta\, a_2)b_2 &= 0 ,\\
  a_2(b_1+\gamma\, b_2) &=0.
 \end{cases}
\end{align}
If $b_2=0$, then by $b\neq0$ we must have $b_1\neq0$ and we obtain $a_1=a_2=0$,  meaning that the operator $\mathbb{A}$ is $\mathbb{C}$-elliptic. Otherwise, with $b_2\neq0$ we have $a_1+\beta\,a_2=0$. If $a_2=0$ then we are done. We show, that the case $a_2\neq0$ cannot occur. Indeed, if $a_2\neq0$ then \eqref{eq:cases} yields:
\begin{align*}
 \begin{cases}
  b_1+\gamma\, b_2 &= 0\\
  -\beta\,b_1+\alpha\,b_2&=0
 \end{cases} \quad \Leftrightarrow\quad \begin{pmatrix} 1 & \gamma \\ -\beta & \alpha \end{pmatrix} b = 0
\end{align*}
which cannot be fulfilled since $b\neq0$ and $\alpha+\beta\gamma\neq0$.
\end{proof}
Let us now be more precise and explain in detail where the proof of Proposition \ref{prop:dodgethedodo} fails if $\A$ is elliptic but \textbf{not} $\mathbb{C}$-elliptic. Ellipticity implies that for any $\xi\in\R^{2}$, $\dim(\A[\xi](\R^{2}))=2$ and by the foregoing lemma we must have $\dim(\mathscr{A}[\R^{2\times 2}])=2$. Hence, if $\A$ is not $\mathbb{C}$-elliptic, then precisely two pure $\A$-tensors $\mathbf{e}_{i}\otimes_{\A}\mathbf{e}_{j}$, $(i,j)\in\mathcal{I}\coloneqq \{(i_{0},j_{0}),(i_{1},j_{1})\}$,  are linearly independent. In order to introduce the linear map $\mathcal{L}$, we define it via its action on these basis vectors. At a first glance, this seems to only yield two conditions, but this is not so because of~\eqref{eq:frankrosin} as we here pose compatibility conditions for \emph{all} indices that ought to be fulfilled. In particular, the definition of $\mathcal{L}$ on the fixed basis vectors must be compatible with~\eqref{eq:frankrosin} also for $(i,j)\notin\mathcal{I}$, and since~\eqref{eq:frankrosin} includes all rank-one-matrices, this is non-trivial. As can be seen explicitly from Example~\ref{ex:tensorspanTFsymgradn=2} below, non-$\mathbb{C}$-elliptic operators do not satisfy this property. 
\begin{example}[Gradient and symmetric gradient]
For the ($\mathbb{C}$-elliptic) gradient, the set $\{\mathbf{e}_{i}\otimes_{\A}\mathbf{e}_{j}\}$ is just 
\begin{align*}
\left\{\begin{pmatrix} 1 & 0 \\ 0 & 0 \end{pmatrix},\;\begin{pmatrix} 0 & 1 \\ 0 & 0 \end{pmatrix},\;\begin{pmatrix} 0 & 0 \\ 1 & 0 \end{pmatrix},\;\begin{pmatrix} 0 & 0 \\ 0 & 1 \end{pmatrix}\right\}
\end{align*}
and so the gradient falls into the case $\dim(\mathscr{A}[\R^{2\times 2}])=4$ in the case distinction of the above proof. For the ($\mathbb{C}$-elliptic) deviatoric or symmetric gradients, the set $\{\mathbf{e}_{i}\otimes_{\A}\mathbf{e}_{j}\}$ is just 
\begin{align*}
& \left\{\begin{pmatrix} \frac{1}{2} & 0 \\ 0 & -\frac{1}{2} \end{pmatrix},\;\begin{pmatrix} 0 & 1 \\ 0 & 0 \end{pmatrix},\;\begin{pmatrix} 0 & 0 \\ 1 & 0 \end{pmatrix},\;\begin{pmatrix} -\frac{1}{2} & 0 \\ 0 & \frac{1}{2} \end{pmatrix}\right\} &\text{(deviatoric gradient)},\\
&\left\{\begin{pmatrix} 1 & 0 \\ 0 & 0 \end{pmatrix},\;\begin{pmatrix} 0 & \frac{1}{2} \\ \frac{1}{2} & 0 \end{pmatrix},\;\begin{pmatrix} 0 & 0 \\ 0 & 1 \end{pmatrix}\right\} &\text{(symmetric gradient)},
\end{align*}
and so both the deviatoric and the symmetric gradient falls into the case $\dim(\mathscr{A}[\R^{2\times 2}])=3$ in the case distinction of the above proof. 
\end{example}
\begin{example}[Trace-free symmetric gradient]\label{ex:tensorspanTFsymgradn=2}
For the (non-$\mathbb{C}$-elliptic) trace-free symmetric gradient, the set $\{\mathbf{e}_{i}\otimes_{\A}\mathbf{e}_{j}\}$ is given by 
\begin{align*}
\left\{\begin{pmatrix} \frac{1}{2} & 0 \\ 0 & -\frac{1}{2} \end{pmatrix},\;\begin{pmatrix} 0 & \frac{1}{2} \\ \frac{1}{2} & 0 \end{pmatrix},\;\begin{pmatrix} -\frac{1}{2} & 0 \\ 0 & \frac{1}{2} \end{pmatrix}\right\}=:\{A,B,C\}, 
\end{align*}
and so the first and the third element are linearly dependent. We now verify explicitly that the relation~\eqref{eq:frankrosin} cannot be achieved in this case for some $\Q\in\mathrm{GL}(2)$. Suppose towards a contradiction that~\eqref{eq:frankrosin} can be achieved. Then 
\begin{align}\label{eq:hanshaas1}
\begin{split}
\begin{pmatrix} 0 & 0 \\ 0 & 1 \end{pmatrix} - \gamma_{22}\Q   & \stackrel{\eqref{eq:frankrosin}}{=} \mathcal{L}(\mathbf{e}_{2}\otimes_{\A}\mathbf{e}_{2}) \stackrel{A=-C}{=} -\mathcal{L}(\mathbf{e}_{1}\otimes_{\A}\mathbf{e}_{1})   \stackrel{\eqref{eq:frankrosin}}{=} \gamma_{11}\Q - \begin{pmatrix} 1 & 0 \\ 0 & 0 \end{pmatrix}
\end{split}
\end{align}
and 
\begin{align}\label{eq:hanshaas2}
\begin{split}
\begin{pmatrix} 0 & 1 \\ 0 & 0 \end{pmatrix} - \gamma_{12}\Q   & \stackrel{\eqref{eq:frankrosin}}{=} \mathcal{L}(\mathbf{e}_{1}\otimes_{\A}\mathbf{e}_{2}) = \mathcal{L}(\mathbf{e}_{2}\otimes_{\A}\mathbf{e}_{1})   \stackrel{\eqref{eq:frankrosin}}{=} \begin{pmatrix} 0 & 0 \\ 1 & 0 \end{pmatrix} - \gamma_{21}\Q.
\end{split}
\end{align}
From~\eqref{eq:hanshaas1} we infer that $\Q$ needs to be a multiple of the identity matrix, whereas we infer from~\eqref{eq:hanshaas2} that $\Q$ is contained in the skew-symmetric matrices, and this is impossible. 
\end{example}
\subsection{{The implication '$\ref{item:KMScriticalMain2}\Rightarrow\ref{item:KMScriticalMain1} $'} of Theorem~\ref{thm:main2}}\label{sec:suffpart}
Proposition~\ref{prop:dodgethedodo} crucially allows to adapt the reduction to \textsc{Bourgain-Brezis}-type estimates as pursued in \cite{GLP} in the symmetric gradient case. We state a version of the result that will prove useful in later sections too:
\begin{lemma}[{\cite{BoBre,BrezisVS,Val}}]\label{lem:BoBre} Let $n\geq 2$. Then the following hold: 
\begin{enumerate}
\item\label{item:BoBre1} There exists a constant $c_{n}>0$ such that 
\begin{align}\label{eq:BourgainBrezis}
\norm{f}_{\dot{\sobo}^{-1,\frac{n}{n-1}}(\R^{n})}\leq c_{n}\Big(\norm{\div f}_{\dot{\sobo}^{-2,\frac{n}{n-1}}(\R^{n})}+\norm{f}_{\lebe^{1}(\R^{n})} \Big)
\end{align}
holds for all $f\in\hold_{c}^{\infty}(\R^{n};\R^{n})$. Moreover, if $\Omega\subset\R^{n}$ is an open and bounded domain with Lipschitz boundary, estimate~\eqref{eq:BourgainBrezis} persists for all $f\in\hold^{\infty}(\Omega;\R^{n})$ with $\R^{n}$ being replaced by $\Omega$. 
\item\label{item:BoBre2} If $n\geq 3$, there exists a constant $c_{n}>0$ such that 
\begin{align}\label{eq:BourgainBrezis1}
\norm{f}_{\lebe^{\frac{n}{n-1}}(\R^{n})}\leq c_{n}\norm{\curl f}_{\lebe^{1}(\R^{n})}\qquad\text{for all}\;f\in\hold_{c,\div}^{\infty}(\R^{n};\R^{n}). 
\end{align}
\end{enumerate}
\end{lemma}
For completeness, let us note that by embedding $\hold_{c}^{\infty}(\R^{n};\R^{n})$ into $\sobo^{-1,p'}(\R^{n};\R^{n})$ via the dual pairing $\langle f,\varphi\rangle_{\dot{\sobo}^{-1,p'}\times\dot{\sobo}^{1,p}} = \int_{\R^{n}}\langle f,\varphi\rangle_{\R^n}\dif x$, for any $1<p<\infty$ and $A\in\R^{n\times n}$ there exists $c=c(p,n,A)>0$ such that we have the elementary inequality
\begin{align}\label{eq:normo}
\norm{Af}_{\dot{\sobo}^{-1,p'}(\R^{n})}\leq c\norm{f}_{\dot{\sobo}^{-1,p'}(\R^{n})}\qquad\text{for all}\;f\in\hold_{c}^{\infty}(\R^{n};\R^{n}).
\end{align}
Based on the preceding lemma and the results from the previous section we may now pass on to \begin{proof}[Proof of {the implication '$\ref{item:KMScriticalMain2}\Rightarrow\ref{item:KMScriticalMain1} $' of Theorem~\ref{thm:main2}}]
Let $P\in\hold_{c}^{\infty}(\R^{2};\R^{2\times 2})$. By Proposition~\ref{prop:dodgethedodo}, there exists a linear map $\mathcal{L}\colon\R^{2\times 2}\to\R^{2\times 2}$, $\Q\in\mathrm{GL}(2)$ and a linear function $\gamma\colon \R^{2\times 2}\to\R$ such that 
\begin{align}\label{eq:police1}
P=\mathcal{L}(\mathscr{A}[P])+\gamma(P)\Q\qquad\text{everywhere in $\R^{2}$}. 
\end{align}
Realising that 
\begin{align}\label{eq:CurlskalarQ}
\Curl(\gamma(P)\Q)=\Q\begin{pmatrix} -\partial_{2}\gamma(P)\\ \partial_{1}\gamma(P)\end{pmatrix},
\end{align}
the pointwise relation \eqref{eq:police1} gives 
\begin{align*}
\Curl P=\Curl(\mathcal{L}(\mathscr{A}[P]))+\Q\begin{pmatrix} -\partial_{2}\gamma(P)\\ \partial_{1}\gamma(P)\end{pmatrix}
\end{align*}
whence the invertibility of $\Q$ yields 
\begin{align}\label{eq:changes}
\f \coloneqq  \Q^{-1}\Curl P=\Q^{-1}\Curl(\mathcal{L}(\mathscr{A}[P]))+\begin{pmatrix} -\partial_{2}\gamma(P)\\ \partial_{1}\gamma(P)\end{pmatrix}. 
\end{align}
In consequence, 
\begin{align}\label{eq:hlesch}
\begin{split}
\div \f  = \div\begin{pmatrix} \mathbb{L}_{1}(\mathscr{A}[P])\\ \mathbb{L}_{2}(\mathscr{A}[P]) \end{pmatrix} = \mathbb{B}(\mathscr{A}[P])
\end{split}
\end{align}
for some scalar, linear, homogeneous second order differential operator $\mathbb{B}$. Now, based on inequality~\eqref{eq:BourgainBrezis} and recalling $n=2$, we infer 
\begin{align}
\norm{\f}_{\dot{\sobo}^{-1,2}(\R^{2})} & \leq c\,\Big(\norm{\div \f}_{\dot{\sobo}^{-2,2}(\R^{2})} + \norm{\f}_{\lebe^{1}(\R^{2})}\Big) \notag\\ 
& \!\!\!\stackrel{\eqref{eq:hlesch}}{=} c\,\Big(\norm{\mathbb{B}(\mathscr{A}[P])}_{\dot{\sobo}^{-2,2}(\R^{2})} + \norm{\f}_{\lebe^{1}(\R^{2})} \Big)\notag\\ 
& \leq c\,\Big(\norm{\D^{2}(\mathscr{A}[P])}_{\dot{\sobo}^{-2,2}(\R^{2})} + \norm{\f}_{\lebe^{1}(\R^{2})} \Big)\notag\\ 
&\leq c\,\Big(\norm{\mathscr{A}[P]}_{\lebe^{2}(\R^{2})} + \norm{\f}_{\lebe^{1}(\R^{2})} \Big). 
\end{align}
By definition of $\f$, cf.~\eqref{eq:changes}, the previous inequality implies 
\begin{align}\label{eq:construct}
\norm{\Q^{-1}\Curl P}_{\dot{\sobo}^{-1,2}(\R^{2})}\leq c\,\Big(\norm{\mathscr{A}[P]}_{\lebe^{2}(\R^{2})} + \norm{\Q^{-1}\Curl P}_{\lebe^{1}(\R^{2})} \Big), 
\end{align}
and~\eqref{eq:normo} then yields by virtue of $\Q\in\mathrm{GL}(2)$
\begin{align}\label{eq:shinichi}
\begin{split}
\norm{\Curl P}_{\dot{\sobo}^{-1,2}(\R^{2})} & = \norm{\Q(\Q^{-1}\Curl P)}_{\dot{\sobo}^{-1,2}(\R^{2})} \\ & \!\!\!\!\stackrel{\eqref{eq:construct}}{\leq} c(\Q) \Big(\norm{\mathscr{A}[P]}_{\lebe^{2}(\R^{2})} + \norm{\Curl P}_{\lebe^{1}(\R^{2})} \Big). 
\end{split}
\end{align}
Next, put $f\coloneqq \Curl P\in\hold_{c}^{\infty}(\R^{2};\R^{2})$ and consider the solution $u=(u_{1},u_{2})^{\top}$ of the equation $-\Delta u= f$ obtained as $u=\Phi_{2}*f$, where $\Phi_{2}$ is the two-dimensional Green's function for the negative Laplacian. By classical elliptic regularity estimates we have on the one hand
\begin{align}\label{eq:ellipticsimple}
\norm{\D u}_{\lebe^{2}(\R^{2})} \leq c\,\norm{\Delta u}_{{\dot{\sobo}}{^{-1,2}}(\R^{2})} = c\norm{f}_{{\dot{\sobo}}{^{-1,2}}(\R^{2})}.
\end{align}
On the other hand, setting 
\begin{align*}
T\coloneqq \begin{pmatrix}\partial_{2}u_{1} & -\partial_{1}u_{1} \\ \partial_{2}u_{2} &-\partial_{1}u_{2} \end{pmatrix},
\end{align*}
we find by \eqref{eq:ellipticsimple} and \eqref{eq:shinichi}
\begin{align}\label{eq:quittung}
\norm{T}_{\lebe^{2}(\R^{2})} & = \norm{\D u}_{\lebe^{2}(\R^{2})} \leq c(\Q) \Big(\norm{\mathscr{A}[P]}_{\lebe^{2}(\R^{2})} + \norm{\Curl P}_{\lebe^{1}(\R^{2})} \Big)
\end{align}
and 
\begin{align*}
\Curl(T-P) = \begin{pmatrix} -\Delta u_{1} \\ -\Delta u_{2}\end{pmatrix}- f = 0. 
\end{align*}
We may thus write $T-P=\D v$ for some $v\colon\R^{2}\to\R^{2}$, for which Lemma~\ref{lem:CZK} yields\footnote{Note that $T$ is not compactly supported (and neither is $v$), but $v\in{\dot{\sobo}}{^{1,2}}(\R^{2};\R^{2})$ and then the Korn-type inequality underlying the first inequality in~\eqref{eq:kornesque} follows by smooth approximation,  directly using the definition of ${\dot{\sobo}}{^{1,2}}(\R^{2})$.}
\begin{align}\label{eq:kornesque}
\begin{split}
\norm{T-P}_{\lebe^{2}(\R^{2})} &= \norm{\D v}_{\lebe^{2}(\R^{2})} \leq c \norm{\A v}_{\lebe^{2}(\R^{2})} = c\norm{\mathscr{A}[\D v]}_{\lebe^{2}(\R^{2})} \\ 
& \leq c\norm{\mathscr{A}[T-P]}_{\lebe^{2}(\R^{2})} \leq c \norm{T}_{\lebe^{2}(\R^{2})}+c\norm{\mathscr{A}[P]}_{\lebe^{2}(\R^{2})} .
\end{split}
\end{align}
The proof is then concluded by splitting 
\begin{align*}
\norm{P}_{\lebe^{2}(\R^{2})} & \leq \norm{P-T}_{\lebe^{2}(\R^{2})} + \norm{T}_{\lebe^{2}(\R^{2})}
\end{align*}
and using~\eqref{eq:kornesque},~\eqref{eq:quittung} for the first and~\eqref{eq:quittung} for the second term. The proof is complete. 
\end{proof}
\begin{remark}
 Note that, despite of the special situation in three dimensions, in $n=2$ dimensions we have in general 
 \begin{equation}
  \div \Curl P = -\partial_{12}(P_{11}-P_{22})+\partial_{11}P_{12}-\partial_{22}P_{21}\not\equiv 0.
 \end{equation}
In fact, it is a crucial step in the previous proof to express $\div\Curl P$ by a linear combination of second derivatives of $\mathscr{A}[P]$. This is clearly the case if $\mathscr{A}[P]=\dev P$ but is not fulfilled for general $\mathbb{C}$-elliptic operators $\mathbb{A}$. For this reason we need Proposition \ref{prop:dodgethedodo} to ensure that there exists an invertible matrix $\Q$ such that $
 \div \Q^{-1}\Curl P = \mathcal{L}(\D^2 \mathscr{A}[P])$. 
\end{remark}

\subsection{{The implication '$\ref{item:KMScriticalMain1}\Rightarrow\ref{item:KMScriticalMain2} $' of Theorem~\ref{thm:main2}}}\label{sec:necpart}
Again, let $n=2$ and $p=1$. As in \cite{GmSp}, one directly obtains the necessity of $\mathscr{A}$ being induced by an elliptic differential operator for~\eqref{eq:KMS2} to hold; in fact, testing~\eqref{eq:KMS2} with gradient fields $P=\D u$ gives us the usual Korn-type inequality $\norm{\D u}_{\lebe^{2}(\R^{2})}\leq c\norm{\A u}_{\lebe^{2}(\R^{2})}$ for all $u\in\hold_{c}^{\infty}(\R^{2};\R^{2})$. Then one uses Lemma~\ref{lem:CZK} to conclude the ellipticity; hence, in all of the following we may tacitly assume $\mathscr{A}$ to be induced by an elliptic differential operator. 

The necessity of $\mathbb{C}$-ellipticity requires a  refined argument, which we address now: {
\begin{lemma}\label{lem:superhelpful}
In the situation of Theorem \ref{thm:main2}, validity of \eqref{eq:KMS2} implies that $\mathscr{A}$ induces a $\mathbb{C}$-elliptic differential operator. 
\end{lemma}}
\begin{proof} Suppose that $\A$ is not $\mathbb{C}$-elliptic, meaning that there exist $\xi\in\mathbb{C}^{2}\backslash\{0\}$ and $v=\Re(v)+\imag\Im(v)\in\mathbb{C}^{2}\backslash\{0\}$ such that $\A[\xi]v = 0$. Writing this last equation by separately considering real and imaginary parts, we obtain 
\begin{align}\label{eq:telephoneline}
\begin{split} 
& \A[\Re(\xi)]\Re(v) = \A[\Im(\xi)]\Im(v), \\
& \A[\Im(\xi)]\Re(v) = - \A[\Re(\xi)]\Im(v).
\end{split}
\end{align}
For $f\in\hold_{c}^{\infty}(\R^{2})$, we then define $x_{\xi}\coloneqq (\langle x,\Re(\xi)\rangle,\langle x,\Im(\xi)\rangle)$ and 
\begin{align}\label{eq:Pfdef}
\begin{split}
P_{f}(x) \coloneqq&\  \Re(v)\otimes \Re(\xi) (\partial_{1}f)(x_\xi)- \Im(v)\otimes \Im(\xi)(\partial_{1}f)(x_\xi) \\ 
&  + \Re(v)\otimes\Im(\xi)(\partial_{2}f)(x_\xi) + \Im(v)\otimes\Re(\xi)(\partial_{2}f)(x_\xi). 
\end{split}
\end{align}
We first claim that there exists a constant $c>0$ such that 
\begin{align}\label{eq:lowerbound1}
\norm{\nabla f}_{\lebe^{2}(\R^{2})}\leq c\,\norm{P_{f}}_{\lebe^{2}(\R^{2})}. 
\end{align}
To this end, we first note that we may assume $\Re(\xi)$ and $\Im(\xi)$ to be linearly independent over $\R$. This argument is certainly clear to experts, but since it is crucial for our argument, we give the quick proof. Suppose that  $\Re(\xi)$ and $\Im(\xi)$ are not linearly independent. Then there are three options: 
\begin{itemize}
\item[(i)] If $\Im(\xi)=0$, then $\Re(\xi)\neq 0$ (as otherwise $\xi=0$). Then~$\eqref{eq:telephoneline}_{1}$ implies $\Re(v)=0$ by ellipticity of $\A$, and~$\eqref{eq:telephoneline}_{2}$ implies $\Im(v)=0$ by ellipticity of $\A$. Thus $v=0\in\mathbb{C}^{2}$, contradicting our assumption $v\in\mathbb{C}^{2}\backslash\{0\}$. 
\item[(ii)] If $\Re(\xi)=0$, we may imitate the argument from (i) to arrive at a contradiction. 
\item[(iii)] Based on (i) and (ii), we may assume that $\Im(\xi),\Re(\xi)\neq 0$ and that there exists $\lambda\neq 0$ such that $\Re(\xi)=\lambda\Im(\xi)$. Inserting this relation into \eqref{eq:telephoneline}, we arrive at 
\begin{align*}
\lambda\A[\Im(\xi)]\Re(v) &= \A[\Im(\xi)]\Im(v) && \Rightarrow& \A[\Im(\xi)](\lambda\Re(v)-\Im(v))&=0,\\
\A[\Im(\xi)]\Re(v) &= -\lambda\A[\Im(\xi)]\Im(v)&& \Rightarrow &\A[\Im(\xi)](\Re(v)+\lambda\Im(v))&=0. 
\end{align*}
By our assumption, $\lambda\neq 0$, and ellipticity of $\A$ implies that 
\begin{align*}
\lambda\Re(v)=\Im(v)\;\text{and}\;\Re(v)+\lambda\Im(v)=0,\;\;\;\text{so}\;(1+\lambda^{2})\Re(v)=0, 
\end{align*}
so $\Re(v)=\Im(v)=0$, and this is at variance with our assumption $v\neq 0$. 
\end{itemize} 
Secondly, we similarly note that $\Re(v)$ and $\Im(v)$ can be assumed to be linearly independent over $\R$. Suppose that this is not the case. Then there are three options:
\begin{itemize}
\item[(i')] If $\Re(v)=0$, then necessarily $\Im(v)\neq 0$. Then $\eqref{eq:telephoneline}_{1}$ implies that $\Im(\xi)=0$. Inserting this into $\eqref{eq:telephoneline}_{2}$ yields $\Re(\xi)=0$ and so $\xi=0$, which is at variance with our assumption of $\xi\in\mathbb{C}^{2}\backslash\{0\}$. 
\item[(ii')] If $\Im(v)=0$, we may imitate the argument from (i') to arrive at a contradiction. 
\item[(iii')] Based on (i') and (ii'), we may assume that $\Im(v),\Re(v)\neq 0$ and that there exists $\lambda\neq 0$ such that $\Re(v)=\lambda\Im(v)$.  Inserting this relation into \eqref{eq:telephoneline} yields
\begin{align*}
 \lambda\A[\Re(\xi)]\Im(v) &= \A[\Im(\xi)]\Im(v) && \Rightarrow &\A[\lambda\Re(\xi)-\Im(\xi)]\Im(v)&= 0,\\
 \lambda\A[\Im(\xi)]\Im(v) &= - \A[\Re(\xi)]\Im(v) &&\Rightarrow &\A[\lambda\Im(\xi)+\Re(\xi)]\Im(v)&=0. 
\end{align*}
By our assumption $\Im(v)\neq 0$, ellipticity of $\A$ yields 
\begin{align*}
\lambda\Re(\xi)=\Im(\xi)\;\text{and}\;-\lambda\Im(\xi)=\Re(\xi) \Rightarrow (1+\lambda^{2})\Im(\xi)=0. 
\end{align*}
But then $\Im(\xi)=0$ and, by $\lambda\neq 0$, $\Re(\xi)=0$, so $\xi=0$, in turn being at variance with $\xi\neq 0$. 
\end{itemize} 
Based on (i)--(iii) and (i')--(iii'), we conclude that the set 
\begin{align*}
\{A,B,C,D\}\coloneqq \{\Re(v)\otimes \Re(\xi),\Im(v)\otimes \Im(\xi),\Re(v)\otimes\Im(\xi), \Im(v)\otimes\Re(\xi)\} 
\end{align*}
displays a basis of $\R^{2\times 2}$, and that 
\begin{align*}
|\!|\!| P |\!|\!| \coloneqq  |a|+|b|+|c|+|d|\;\;\;\text{whenever}\; P=aA+bB+cC+dD\in\R^{2\times 2} 
\end{align*}
is a norm on $\R^{2\times 2}$. This norm is equivalent to any other norm on $\R^{2\times 2}$, and hence~\eqref{eq:lowerbound1} follows by the definition of $P_{f}$ and a change of variables, again recalling that $\Re(\xi)$, $\Im(\xi)$ and $\Re(v)$, $\Im(v)$ are linearly independent. 

In a next step, we record that for any $f\in\hold_{c}^{\infty}(\R^{2})$ 
\begin{align}\label{eq:upperbound1}
\begin{split}
\mathscr{A}[P_{f}(x)] & = \A[\Re(\xi)]\Re(v)(\partial_{1}f)(x_{\xi}) - \A[\Im(\xi)]\Im(v)(\partial_{1}f)(x_{\xi}) \\ 
& + \A[\Im(\xi)]\Re(v)(\partial_{2}f)(x_{\xi}) + \A[\Re(\xi)]\Im(v)(\partial_{2}f)(x_{\xi}) \stackrel{\eqref{eq:telephoneline}}{=} 0. 
\end{split} 
\end{align}
To conclude the proof, we will now establish that the Korn-Maxwell-Sobolev inequality in this situation yields the contradictory estimate
\begin{align}\label{eq:OlliNorthContra}
\norm{\nabla f}_{\lebe^{2}(\R^{2})} \leq c\,\norm{\Delta f}_{\lebe^{1}(\R^{2})}\qquad\text{for all}\;f\in\hold_{c}^{\infty}(\R^{2}). 
\end{align}
To this end, note that for any $\varphi\in\hold_{c}^{\infty}(\R^{2})$ and any $X\in\R^{2\times 2}$ we have as in \eqref{eq:CurlskalarQ}:
\begin{align}\label{eq:CurlIdentityNeccPart}
\Curl(\varphi\, X)=X \binom{-\partial_{2}\varphi}{\partial_{1}\varphi}.
\end{align}
Moreover, writing $\Re(\xi)=(\xi_{11},\xi_{12})^{\top}$ and $\Im(\xi)=(\xi_{21},\xi_{22})^{\top}$, we find 
\begin{align}\label{eq:enjoythetaste}
\begin{split}
\partial_{1}\big((\partial_{1}f)(x_{\xi}) \big) = \left\langle  (\nabla\partial_{1}f)(x_{\xi}),\binom{\xi_{11}}{\xi_{21}}\right\rangle, \\ 
\partial_{2}\big((\partial_{1}f)(x_{\xi}) \big) = \left\langle  (\nabla\partial_{1}f)(x_{\xi}),\binom{\xi_{12}}{\xi_{22}}\right\rangle, \\  
\partial_{1}\big((\partial_{2}f)(x_{\xi})\big) = \left\langle  (\nabla\partial_{2}f)(x_{\xi}),\binom{\xi_{11}}{\xi_{21}}\right\rangle,\\  \partial_{2}\big((\partial_{2}f)(x_{\xi})\big) = \left\langle  (\nabla\partial_{2}f)(x_{\xi}),\binom{\xi_{12}}{\xi_{22}}\right\rangle. 
\end{split}
\end{align}
As a consequence of~\eqref{eq:enjoythetaste}, we have for $j\in\{1,2\}$
\begin{align}\label{eq:fensterputzer}
\binom{-\partial_{2}( (\partial_{j}f)(x_{\xi}))}{\partial_{1}((\partial_{j}f)(x_{\xi}))} =\begin{pmatrix} -\xi_{12} & -\xi_{22} \\ \xi_{11} & \xi_{21}\end{pmatrix}(\nabla \partial_{j}f)(x_{\xi}),
\end{align}
Based on~\eqref{eq:CurlIdentityNeccPart}, the previous two identities imply by definition of $P_{f}$ (cf.~\eqref{eq:Pfdef})
\begin{align*} 
\Curl P_{f}(x) & = \Re(v)\,(\xi_{11},\xi_{12})\,\begin{pmatrix} -\xi_{12} & -\xi_{22} \\ \xi_{11} & \xi_{21}\end{pmatrix}(\nabla \partial_{1}f)(x_{\xi})\\ 
& - \mathrm{Im}(v)\,(\xi_{21},\xi_{22}) \begin{pmatrix} -\xi_{12} & -\xi_{22} \\ \xi_{11} & \xi_{21}\end{pmatrix}(\nabla \partial_{1}f)(x_{\xi}) \\ 
& + \Re(v)\,(\xi_{21},\xi_{22})\begin{pmatrix} -\xi_{12} & -\xi_{22} \\ \xi_{11} & \xi_{21}\end{pmatrix}(\nabla \partial_{2}f)(x_{\xi}) \\ 
& + \Im(v)\,(\xi_{11},\xi_{12})\begin{pmatrix} -\xi_{12} & -\xi_{22} \\ \xi_{11} & \xi_{21}\end{pmatrix}(\nabla \partial_{2}f)(x_{\xi}). 
\end{align*} 
Now define $\displaystyle \alpha \coloneqq  \det\begin{pmatrix} \xi_{11} & \xi_{12} \\ \xi_{21} & \xi_{22}\end{pmatrix} $ so that 
\begin{align} 
\Curl P_{f}(x) & = \Re(v)\,(0,-\alpha)(\nabla \partial_{1}f)(x_{\xi}) - \Im(v)\,(\alpha,0)(\nabla \partial_{1}f)(x_{\xi}) \notag\\ 
& \quad + \Re(v)\,(\alpha,0)(\nabla\partial_{2}f)(x_{\xi}) +\Im(v)\,(0,-\alpha)(\nabla\partial_{2}f)(x_{\xi}) \notag\\ 
& = -\alpha\Re(v) (\partial_{21}f)(x_{\xi}) -\alpha\Im(v)(\partial_{11}f)(x_{\xi})\notag\\
 & \quad +\alpha\Re(v)(\partial_{12}f)(x_{\xi})-\alpha\Im(v)(\partial_{22}f)(x_{\xi}) \notag\\ 
& = -\alpha (\Delta f)(x_{\xi})\,\Im(v). \label{eq:upperbound2}
\end{align}
Hence, in view of \eqref{eq:lowerbound1}, \eqref{eq:upperbound1} and \eqref{eq:upperbound2}, starting with the Korn-Maxwell-Sobolev inequality and a change of variables (recall that $\Re(\xi),\Im(\xi)$ are linearly independent), we end up at the contradictory estimate \eqref{eq:OlliNorthContra}. Thus, $\A$ has to be $\mathbb{C}$-elliptic and the proof is complete. 
\end{proof}
{\subsection{Proof of Theorem \ref{thm:main2}} We now briefly pause to concisely gather the arguments required for the proof of Theorem \ref{thm:main2}. The direction '\ref{item:KMScriticalMain1}$\Rightarrow$\ref{item:KMScriticalMain2}' is given by Lemma \ref{lem:superhelpful}, whereas the equivalence of \ref{item:KMScriticalMain2} and \ref{item:KMScriticalMain3} is a direct consequence of Lemma \ref{lem:helpfulDiffOp}. In turn, the  equivalence '\ref{item:KMScriticalMain2}$\Leftrightarrow$\ref{item:KMScriticalMain4}' is established in Lemma \ref{lem:characterC}. Finally, the remaining implication '\ref{item:KMScriticalMain2}$\Rightarrow$ \ref{item:KMScriticalMain1}' is proved at the end of Paragraph \ref{sec:suffpart} above. In conclusion, the proof of Theorem \ref{thm:main2} is complete.}
\subsection{KMS inequalities with non-zero boundary values}
Even though the present paper concentrates on compactly supported maps, we like to comment on how Proposition~\ref{prop:dodgethedodo} can be used to derive variants for maps with non-zero boundary values; its generalisation to higher space dimensions shall be pursued elsewhere. 
\begin{theorem}[Non-zero boundary values, $(p,n)=(1,2)$]\label{thm:domains}
The following are equivalent: 
\begin{enumerate}
\item\label{item:bdd1} The linear map $\mathscr{A}\colon\R^{2\times 2}\to\R^{2\times 2}$ induces a $\mathbb{C}$-elliptic differential operator $\mathbb{A}$ in $n=2$ dimensions via $\A u \coloneqq \mathscr{A}[\D u]$. 
\item\label{item:bdd2} There exists a \emph{finite dimensional subspace} $\mathcal{K}$ of the $\R^{2\times 2}$-valued polynomials such that for any open and bounded, simply connected domain $\Omega\subset\R^{2}$ with Lipschitz boundary $\partial\Omega$ there exists $c=c(\mathscr{A},\Omega)>0$ such that 
\begin{align}\label{eq:korner}
\min_{\Pi\in\mathcal{K}}\norm{P-\Pi}_{\lebe^{2}(\Omega)}\leq c\,\Big(\norm{\mathscr{A}[P]}_{\lebe^{2}(\Omega)}+\norm{\Curl P}_{\lebe^{1}(\Omega)}\Big)
\end{align}
holds for all $P\in\hold^{\infty}(\Omega;\R^{2\times 2})$. 
\end{enumerate}
\end{theorem}
\begin{proof}[Proof of Theorem~\ref{thm:domains}] Ad '\ref{item:bdd1}$\Rightarrow$\ref{item:bdd2}'. 
As {at the end of Paragraph \ref{sec:suffpart}}, we use Proposition~\ref{prop:dodgethedodo} to write $P=\mathcal{L}(\mathscr{A}[P])+\gamma(P)\,\Q$ pointwisely for all $P\in\hold^{\infty}(\Omega;\R^{2\times 2})$, suitable linear maps $\mathcal{L}\colon\R^{2\times 2}\to\R^{2\times 2}$, $\gamma\colon\R^{2\times 2}\to\R$ and a fixed  matrix $\Q\in\mathrm{GL}(2)$. Let $1<q<\infty$. For $\Omega\subset\R^{2}$ has Lipschitz boundary, we may invoke the Ne\v{c}as estimate from Lemma~\ref{lem:Necas} to obtain for all $P\in\hold^{\infty}(\Omega;\R^{2\times 2})$
\begin{align}\label{eq:Necas1}
\begin{split}
\norm{P}_{\lebe^{q}(\Omega)} & \leq c\,\Big(\norm{\mathscr{A}[P]}_{\lebe^{q}(\Omega)}+\norm{\gamma(P)}_{\lebe^{q}(\Omega)} \Big) \\ 
& \leq c\,\Big(\norm{\mathscr{A}[P]}_{\lebe^{q}(\Omega)}+\norm{\gamma(P)}_{\sobo^{-1,q}(\Omega)}+\norm{\nabla\gamma(P)}_{\sobo^{-1,q}(\Omega)} \Big)
\end{split}
\end{align}
with a constant $c=c(q,\Omega)>0$. We then equally have~\eqref{eq:changes}, whereby 
\begin{align}\label{eq:Necas2}
\begin{split}
\norm{\nabla\gamma(P)}_{\sobo^{-1,q}(\Omega)} & =\norm{\binom{-\partial_{2}\gamma(P)}{\partial_{1}\gamma(P)}}_{\sobo^{-1,q}(\Omega)} \\ 
& \!\!\!\!\stackrel{\eqref{eq:changes}}{\leq} c\,\big( \norm{\Curl P}_{\sobo^{-1,q}(\Omega)}+\norm{\Curl(\mathcal{L}(\mathscr{A}[P]))}_{\sobo^{-1,q}(\Omega)}\big) \\ 
& \leq c\,\big( \norm{\Curl P}_{\sobo^{-1,q}(\Omega)}+\norm{\mathscr{A}[P]}_{\lebe^{q}(\Omega)}\big). 
\end{split}
\end{align}
Combining~\eqref{eq:Necas1} and~\eqref{eq:Necas2}, we arrive at 
\begin{align}\label{eq:Necas3}
\norm{P}_{\lebe^{q}(\Omega)}\leq c\,\big(\norm{\mathscr{A}[P]}_{\lebe^{q}(\Omega)}+\norm{P-\mathcal{L}(\mathscr{A}[P])}_{\sobo^{-1,q}(\Omega)}+\norm{\Curl P}_{\sobo^{-1,q}(\Omega)} \big),
\end{align}
where $c=c(q,\mathscr{A},\Omega)>0$. Next let, adopting the notation from~\eqref{eq:Waqspace},  
\begin{align*}
P\in\mathcal{N}\coloneqq \{P'\in\sobo^{\Curl,q}(\Omega)\mid \mathscr{A}[P']=0\;\;\text{and}\;\;\Curl P'=0\}. 
\end{align*}
Since $\Omega\subset\R^{2}$ is simply connected and has Lipschitz boundary, $\Curl P=0$ implies that $P=\D u$ for some $u\in\sobo^{1,q}(\Omega;\R^{2})$. Then $\mathscr{A}[P]=0$ yields $\mathbb{A}u \coloneqq \mathscr{A}[\D u]=0$. In conclusion, $\mathbb{C}$-ellipticity of $\A$ yields by virtue of Lemma~\ref{lem:Poincare} that there exists $m\in\N$ such that $\mathcal{N}\subset\nabla\mathscr{P}_{m}(\R^{2};\R^{2})\hookrightarrow\lebe^{2}(\Omega;\R^{2\times 2})$. We denote $m_{0}\coloneqq \dim(\mathcal{N})$ and choose an $\lebe^{2}$-orthonormal basis $\{\mathbf{g}_{1},...,\mathbf{g}_{m_{0}}\}$ of $\mathcal{N}$; furthermore, slightly abusing notation, we set 
\begin{align*}
\mathcal{N}^{\bot}\coloneqq \left\{P\in\lebe^{1}(\Omega;\R^{2\times 2})\mid  \int_{\Omega}\langle P,\mathbf{g}_{j}\rangle\dif x = 0\;\text{for all}\;j\in\{1,...,m_{0}\} \right\}.
\end{align*}
Our next claim is that there exists a constant  $c=c(q,\mathscr{A},\Omega)>0$ such that 
\begin{align}\label{eq:Necas4}
\norm{P}_{\lebe^{q}(\Omega)}\leq c\,\big(\norm{\mathscr{A}[P]}_{\lebe^{q}(\Omega)}+\sum_{j=1}^{m_{0}}\left\vert\int_{\Omega}\langle\mathbf{g}_{j},P\rangle\dif x\right\vert +\norm{\Curl P}_{\sobo^{-1,q}(\Omega)} \big)
\end{align}
holds for all $P\in\sobo^{\Curl,q}(\Omega)$.  If~\eqref{eq:Necas4} were false, for each $i\in\N$ we would find $P_{i}\in\sobo^{\Curl,q}(\Omega)$ with $\norm{P_{i}}_{\lebe^{q}(\Omega)}=1$ and 
\begin{align}\label{eq:Necas5}
\big(\norm{\mathscr{A}[P_{i}]}_{\lebe^{q}(\Omega)}+\sum_{j=1}^{m_{0}}\left\vert\int_{\Omega}\langle\mathbf{g}_{j},P_{i}\rangle\dif x\right\vert +\norm{\Curl P_{i}}_{\sobo^{-1,q}(\Omega)} \big)<\frac1i. 
\end{align}
Since $1<q<\infty$, the Banach-Alaoglu theorem allows to pass to a non-relabeled subsequence such that $P_{i}\rightharpoonup P$ weakly in $\lebe^{q}(\Omega;\R^{2\times 2})$ for some $P\in\lebe^{q}(\Omega;\R^{2\times 2})$. Lower semicontinuity of norms with respect to weak convergence and~\eqref{eq:Necas5} then imply $\mathscr{A}[P]=0$. Moreover, we have for all $\varphi\in\sobo_{0}^{1,q'}(\Omega;\R^{2})$ that the formal adjoint of $\Curl$ is given by $\Curl^{*} \varphi =\varphi\nMat{\nabla}{n}\in\lebe^{q'}(\Omega;\R^{2\times 2})$ and therefore
\begin{align*}
|\langle\Curl(P-P_{i}),\varphi\rangle_{\sobo^{-1,q}(\Omega)\times\sobo_{0}^{1,q'}(\Omega)}|\leq \left\vert\int_{\Omega}\langle P-P_{i},\Curl^{*} \varphi \rangle\dif x\right\vert\to 0
\end{align*}
as $i\to\infty$ because of $P_{i}\rightharpoonup P$ in $\lebe^{q}(\Omega;\R^{2\times 2})$. Thus, $\Curl P_{i}\stackrel{*}{\rightharpoonup}\Curl P$ in $\sobo^{-1,q}(\Omega;\R^{2})$. 
By lower semicontinuity of norms for weak*-convergence and~\eqref{eq:Necas5}, we have $\Curl P=0$ in $\sobo^{-1,q}(\Omega;\R^{2})$. On the other hand, since the $\mathbf{g}_{j}$'s are polynomials and hence trivially belong to $\lebe^{q'}(\Omega;\R^{2\times 2})$, we deduce that
\begin{align*}
\begin{cases} 
\Curl P=0&\;\text{in}\;\mathscr{D}'(\Omega;\R^{2}), \\ 
\mathscr{A}[P]=0&\;\text{in}\;\lebe^{q}(\Omega;\R^{2\times 2}),\\ 
P\in\mathcal{N}^{\bot}, 
\end{cases} 
\end{align*}
so that $P\in\mathcal{N}\cap\mathcal{N}^{\bot}=\{0\}$. Finally, since $\mathcal{L}\circ\,\mathscr{A}\colon\R^{2\times 2}\to\R^{2\times 2}$ is linear, we have $P_{i}-\mathcal{L}(\mathscr{A}[P_{i}])\rightharpoonup P-\mathcal{L}(\mathscr{A}[P]) = 0$ weakly in $\lebe^{q}(\Omega;\R^{2\times 2})$, and by compactness of the embedding $\lebe^{q}(\Omega;\R^{2\times 2})\hookrightarrow\hookrightarrow\sobo^{-1,q}(\Omega;\R^{2\times 2})$, we may assume that $P_{i}-\mathcal{L}(\mathscr{A}[P_{i}])\to 0$ strongly in $\sobo^{-1,q}(\Omega;\R^{2\times 2})$. Inserting $P_{i}$ into ~\eqref{eq:Necas3}, the left-hand side is bounded below by $1$ whereas the right-hand side tends to zero as $i\to\infty$. This yields the requisite contradiction and~\eqref{eq:Necas4} follows. 

Based on~\eqref{eq:Necas4}, we may conclude the proof as follows. If $P\in\mathcal{N}^{\bot}$, we use Lemma~\ref{lem:BoBre}~\ref{item:BoBre1} for Lipschitz domains and imitate~\eqref{eq:changes} to find with $\f\coloneqq \Q^{-1}\Curl P$
\begin{align*}
\norm{\Curl P}_{\sobo^{-1,2}(\Omega)} & \leq  c\norm{\f}_{\sobo^{-1,2}(\Omega)} \leq c\,\big(\norm{\mathscr{A}[P]}_{\lebe^{2}(\Omega)}+\norm{\f}_{\lebe^{1}(\Omega)} \big). 
\end{align*}
Then~\eqref{eq:Necas4} with $q=2$ becomes 
\begin{align}\label{eq:Necas6}
\norm{P}_{\lebe^{2}(\Omega)}\leq c\,\big(\norm{\mathscr{A}[P]}_{\lebe^{2}(\Omega)}+\norm{\Curl P}_{\lebe^{1}(\Omega)} \big). 
\end{align}
In the general case, we apply inequality~\eqref{eq:Necas6} to $P-\sum_{j=1}^{m_{0}}\int_{\Omega}\langle\mathbf{g}_{j},P\rangle\dif x\,\mathbf{g}_{j}\in\mathcal{N}^{\bot}$. Hence~\ref{item:bdd2} follows.\medskip

Ad~'\ref{item:bdd2}$\Rightarrow$\ref{item:bdd1}'.\quad Inserting gradients $P=\D u$ into \eqref{eq:korner} yields the Korn-type inequality $$\min_{\Pi\in\mathcal{K}}\norm{\D u-\Pi}_{\lebe^{2}(\Omega)}\leq c\norm{\mathbb{A}u}_{\lebe^{2}(\Omega)} \qquad \text{for all $u\in\hold^{\infty}(\Omega;\R^{2})$.}$$  By routine smooth approximation, this inequality extends to all $u\in{\sobo}{^{\A,2}}(\Omega)$. In particular, if $\A u\equiv 0$ in $\Omega$, then $u$ must coincide with a polynomial of fixed degree. Since this is possible  only if $\A$ is $\mathbb{C}$-elliptic by Lemma~\ref{lem:Poincare},~\ref{item:bdd1} follows and the proof is complete. 
\end{proof}

\section{Sharp KMS-inequalities: The case $(p,n)\neq(1,2)$}\label{sec:sharpAllD}
For the conclusions of Theorem \ref{thm:main1} we only have to establish the sufficiency parts; note that the necessity of ellipticity follows as in Section~\ref{sec:necpart}. 

As in \cite{GmSp}, the proof of the sufficiency part is based on a suitable Helmholtz decomposition. In view of the explicit expressions of these parts (cf.~\eqref{eq:HelmholtzN} in the appendix), we have in all dimensions $n\ge2$ for all $a\in\hold_{c}^{\infty}(\R^{n};\R^{n})$ and all $x\in\R^{n}$:
\begin{align}\label{eq:normAdiv}
\begin{split}
 |a_{\curl}(x)| & \le \frac{1}{n\,\omega_n}\int_{\R^n} |x-y|^{1-n}\cdot |\div a(y)|\,\mathrm{d}y= \frac{1}{c_{1,n}\,n\,\omega_n}\,\mathcal{I}_1 (|\div a|)(x), \\
 |a_{\div}(x)| & \le \frac{\sqrt{n-1}}{n\,\omega_n}\int_{\R^n} |x-y|^{1-n}\cdot |\curl a(y)|\,\mathrm{d}y = \frac{\sqrt{n-1}}{c_{1,n}\,n\,\omega_n}\, \mathcal{I}_1 (|\curl a|)(x), 
\end{split}
\end{align}
where we have used $|\nMat{b}{n}|=\sqrt{n-1}\,|b|$ for any $b\in\R^{n}$ as a  direct consequence of~\eqref{eq:frobenius}. Now, we decompose $P\in\hold_{c}^{\infty}(\R^{n};\R^{m\times n})$ also in its divergence-free part $P_{\Div}$ and curl-free part $P_{\Curl}$; for the following, it is important to note that these matrix differential operators act row-wise. Thus we may write $P_{\Curl} = \D u$ for some $u:\R^n\to\R^m$. Then by Lemma \ref{lem:CZK} we have
\begin{align}\label{eq:HelmholtzPart1}
\begin{split}
 \norm{P_{\Curl}}_{\lebe^{p^*}(\R^n)} &=  \norm{\D u}_{\lebe^{p^*}(\R^n)} \le c \norm{\mathbb{A}u}_{\lebe^{p^*}(\R^n)}=c \norm{\mathscr{A}[\D u]}_{\lebe^{p^*}(\R^n)} \\ & = c \norm{\mathscr{A}[P_{\Curl}]}_{\lebe^{p^*}(\R^n)} = c \norm{\mathscr{A}[P-P_{\Div}]}_{\lebe^{p^*}(\R^n)} \\ & \le c \norm{\mathscr{A}[P]}_{\lebe^{p^*}(\R^n)} +  c \norm{P_{\Div}}_{\lebe^{p^*}(\R^n)}.
\end{split}
\end{align}
The remaining proofs then follow after establishing
\begin{equation}\label{eq:desired}
 \norm{P_{\Div}}_{\lebe^{p^*}(\R^n)}\le c\,\norm{\Curl P}_{\lebe^{p}(\R^n)}.
\end{equation}
To this end, we consider the rows of our incompatible field $P=(P^1,\ldots,P^m)^{\top}$. Using Lemma~\ref{lem:FIT} with $s=1$ yields
\begin{align}\label{eq:Helmholtz2}
\begin{split}
 \norm{P_{\Div}}_{\lebe^{p^*}(\R^n)}&\le\sum_{j=1}^m \norm{P^j_{\div}}_{\lebe^{p^*}(\R^n)}\overset{\eqref{eq:normAdiv}}{\le} c\,\sum_{j=1}^m\norm{\mathcal{I}_{1}(|\curl P^j|)}_{\lebe^{p^*}(\R^n)} \\
 &\overset{\mathclap{\text{Lem.~\ref{lem:FIT}}}}{\le}\quad c \sum_{j=1}^m \norm{\curl P^j}_{\lebe^{p}(\R^n)} \le c\norm{\Curl P}_{\lebe^{p}(\R^n)}. 
\end{split}
\end{align}
Combining~\eqref{eq:HelmholtzPart1} and~\eqref{eq:desired} then yields Theorem~\ref{thm:main1} for $1<p<n$. If $n\geq 3$ and $p=1$, estimate~\eqref{eq:desired} is now a consequence of Lemma~\ref{lem:BoBre}~\ref{item:BoBre2}. Indeed, using~\eqref{eq:BourgainBrezis1}, we have
\begin{align}
 \norm{P_{\Div}}_{\lebe^{1^*}(\R^n)}&\le\sum_{j=1}^m \norm{P^j_{\div}}_{\lebe^{1^*}(\R^n)}\overset{\text{Lem. \ref{lem:BoBre}}}{\le} c \sum_{j=1}^m \norm{\curl P^j_{\div}}_{\lebe^{1}(\R^n)}\notag\\ & =c \sum_{j=1}^m \norm{\curl P^j}_{\lebe^{1}(\R^n)} \le c\norm{\Curl P}_{\lebe^{1}(\R^n)}
\end{align}
where in the penultimate step we added $P^j_{\curl}$ and used the fact that $\curl P^j_{\curl}=0$.  Summarising, the proof of the Theorem \ref{thm:main1} is now complete.

\begin{example}[Incompatible Maxwell/div-curl-inequalities]\label{exampleNeffKnees} In the context of time-incremental infinitesimal elastoplasticity \cite{NeffKnees}, the authors investigated the operator $\mathscr{A}\colon \R^{n\times n}\to\R^{n\times n}$ given by:
\begin{equation}
 \mathscr{A}[P]=\mu_c\skew P + \frac{\kappa}{n}\tr P \cdot\text{\usefont{U}{bbold}{m}{n}1}_{n}, \quad \text{with $\mu_c,\kappa\neq0$}.
\end{equation}
In $n=1$ dimensions this operator clearly induces an elliptic differential operator. We show, that the induced operator $\A$ is also elliptic in $n\ge2$ dimensions. To this end consider for $\xi\in\R^n\backslash\{0\}$:
\begin{equation}
 \mathscr{A}[v\otimes\xi]=\A[\xi]v=\mu_c\skew(v\otimes\xi)+\frac{\kappa}{n}\tr(v\otimes\xi)\cdot\text{\usefont{U}{bbold}{m}{n}1}_{n} \overset{!}{=}0.
\end{equation}
Hence,
\begin{equation}
 \skew(v\otimes \xi)=0 \quad \text{and}\quad \tr(v\otimes\xi)=\langle v,\xi\rangle=0. \label{eq:conditions}
\end{equation}
Since the generalized cross product and the skew-symmetric part of the dyadic product consist of the same entries, cf. \cite[Eq.~(3.9)]{Lewintan1}, the first condition here \eqref{eq:conditions}$_1$ is equivalent to
 \begin{equation}
  v\ntimes{n}\xi=0.
 \end{equation}
Furthermore the generalized cross product satisfies the area property, cf. \cite[Eq.~(3.13)]{Lewintan1}, so that 
\begin{equation}
 0 = |v\ntimes{n}\xi|^2=|v|^2|\xi|^2-\langle v,\xi\rangle^2 \overset{\eqref{eq:conditions}_{2}}{=} |v|^2|\xi|^2 \quad \overset{\xi\neq0}{\Rightarrow} \quad v=0,
\end{equation}
meaning that $\mathscr{A}$ induces an elliptic operator $\A$. But since $\dim\mathscr{A}[\R^{2\times2}]=2$ we infer from Lemma \ref{lem:characterC} that $\A$ is not $\mathbb{C}$-elliptic in $n=2$ dimensions.
Hence, our Theorem \ref{thm:main1} and (anticipating the proof from Section~\ref{sec:subcritical} below) Theorem \ref{thm:main3} are applicable to this $\mathscr{A}$:
\begin{align}\label{eq:exampleNeffKnees}
 \norm{P}_{\lebe^{p^*}(\R^n)} &\le c\,\big(\norm{\skew P}_{\lebe^{p^*}(\R^n)}+\norm{\tr P}_{\lebe^{p^*}(\R^n)} + \norm{\Curl P}_{\lebe^p(\R^n)}\big),
 \end{align}
provided $( n\ge3, 1\le p<n )$ or $( n=2, 1< p<2 )$. Inserting $P=\D u$ and using that $\mathrm{skew}\D u$ and the generalised $\curl u$ consist of the same entries,~\eqref{eq:exampleNeffKnees} generalises the usual $\div$-$\curl$-inequality to incompatible fields. If $n=2$ and $p=1$, inequality~\eqref{eq:exampleNeffKnees} fails since $\A$ fails to be $\mathbb{C}$-elliptic. In fact, the induced elliptic operator $\A$ is not cancelling in any dimension $n\in\N$ since: \begin{equation}
 \bigcap_{\xi\in\R^n\backslash\{0\}} \A[\xi](\R^n)\supseteq\R\cdot\text{\usefont{U}{bbold}{m}{n}1}_{n}.
\end{equation}
\end{example}
We conclude this section by addressing inequalities that (i) involve other function space scales and (ii) also deal with the case $p\geq n$ left out so far. As such, we provide three exemplary results on the validity of inequalities 
\begin{align}\label{eq:Xscale}
\norm{P}_{X(\R^{n})}\leq c\,\big(\norm{\mathscr{A}[P]}_{X(\R^{n})}+\norm{\Curl P}_{Y(\R^{n})} \big),\qquad P\in\hold_{c}^{\infty}(\R^{n};\R^{n\times n})
\end{align}
for $n\geq 2$ that can be approached analogously to~\eqref{eq:normAdiv}--\eqref{eq:Helmholtz2} by combining the Helmholtz decomposition with suitable versions of the fractional integration theorem. To state these results, we require some notions from the theory of weighted Lebesgue and Orlicz spaces, and to keep the paper self-contained, the reader is directed to  Appendix~\ref{sec:appendixweighted} for the underlying background terminology. For the following, let $\A u:=\mathscr{A}[\D u]$ be elliptic.

\textbullet\; \emph{Weighted Lebesgue spaces.} Let $n\geq 2$ and $w$ be a weight that belongs to the Muckenhoupt class $A_{1}$ on $\R^{n}$ (see e.g.~\cite{Muckenhoupt1}) and denote $\lebe^{p}(\R^{n};w)$ the corresponding weighted Lebesgue space. Given $1<p<n$, the Korn-type inequality $\norm{\D u}_{\lebe^{p^{*}}(\R^{n},w)}\leq c\,\norm{\A u}_{\lebe^{p^{*}}(\R^{n},w)}$ for $u\in\hold_{c}^{\infty}(\R^{n};\R^{n})$ (see e.g. \cite{Kalamajska1,Kalamajska} or \cite[Cor.~4.2]{ContiGmeineder}) and $\mathcal{I}_{1}\colon \lebe^{p}(\R^{n};w^{p/p^{*}})\to \lebe^{p^{*}}(\R^{n};w)$ boundedly (see~\cite[Thm.~4]{Muckenhoupt2}) combine to~\eqref{eq:Xscale} with $X(\R^{n})=\lebe^{p^{*}}(\R^{n},w)$ and $Y(\R^{n})=\lebe^{p}(\R^{n};w^{p/p^{*}})$. 

\textbullet\; \emph{Orlicz spaces.} Let $A,B\colon[0,\infty)\to[0,\infty)$ be two Young functions such that $B$ is of class $\Delta_{2}$ and $\nabla_{2}$. Adopting the notation from~\eqref{eq:CianchiDef1} and~\eqref{eq:CianchiDef2}, let us further suppose that $A$ dominates $B_{(n)}$ globally and that $A^{(n)}$ dominates $B$ globally, together with 
\begin{align}\label{eq:CianchiMain}
\int_{0}\frac{B(t)}{t^{1+\frac{n}{n-1}}}\dif t<\infty\;\;\;\text{and}\;\;\;\int_{0}\frac{\widetilde{A}(t)}{t^{1+\frac{n}{n-1}}}\dif t<\infty.
\end{align}
Under these assumptions, \textsc{Cianchi} \cite[Thm.~2(ii)]{CianchiFIT} showed that $\mathcal{I}_{1}\colon\lebe^{A}(\R^{n})\to\lebe^{B}(\R^{n})$ is a bounded linear operator. In conclusion, if $A,B$ are as above, we may then invoke the Korn-type inequality~\cite[Prop.~4.1]{ContiGmeineder} (or~\cite{BreitDiening} in the particular case of the symmetric gradient) to deduce~\eqref{eq:Xscale} with $X=\lebe^{B}(\R^{n})$ and $Y=\lebe^{A}(\R^{n})$. 

\textbullet\; \emph{$\bmo$ and homogeneous H\"{o}lder spaces for $p\geq n$.} If $p= n$ and $Y(\R^{n})=\lebe^{n}(\R^{n})$, scaling suggests $X=\lebe^{\infty}(\R^{n})$, but for this choice inequality~\eqref{eq:Korn3} is readily seen to be false for general elliptic operators $\A$ by virtue of \textsc{Ornstein}'s Non-Inequality on $\lebe^{\infty}$ \cite{Ornstein}. Let $\norm{\cdot}_{\bmo(\R^{N})}$ denote the $\bmo$-seminorm. We may use the fact that the Korn-type inequality from Lemma~\ref{lem:CZK} also holds in the form $\norm{\D u}_{\bmo(\R^{n})}\leq c\,\norm{\A u}_{\bmo(\R^{n})}$; this follows by realising that the map $\D u\mapsto \A u$ is a multiple of the identity plus a Calder\'{o}n-Zygmund operator of convolution type. As observed by \textsc{Spector} and the first author in the three-dimensional case \cite[Sec.~2.2]{GmSp} and since $\Curl$ as defined in Section~\ref{sec:notation} acts row-wisely, we may combine this with $\mathcal{I}_{1}\colon\lebe^{n}(\R^{n})\to\bmo(\R^{n})$ boundedly. This gives us~\eqref{eq:Xscale} with $X(\R^{n})=\bmo(\R^{n})$ and $Y(\R^{n})=\lebe^{n}(\R^{n})$. If $n<p<\infty$, one may argue similarly to obtain~\eqref{eq:Xscale} with $\norm{\cdot}_{X(\R^{n})}$ being the homogeneous $\alpha=1-\frac{n}{p}$-H\"{o}lder seminorm. Combining this with $\mathcal{I}_{1}\colon\lebe^{p}(\R^{n})\to{\dot{\hold}}{^{0,1-n/p}}(\R^{n})$ boundedly yields~\eqref{eq:Xscale} with $X(\R^{n})={\dot{\hold}}{^{0,1-n/p}}(\R^{n})$ and $Y(\R^{n})=\lebe^{p}(\R^{n})$. 
\begin{remark}[$n=1$] 
We now briefly comment on the case $n=1$ that has been left out so far. In this case, the only part maps $\mathscr{A}\colon\R\to\R$ that induce elliptic operators are (non-zero) multiples of the identity. Since every $\varphi\in\hold_{c}^{\infty}(\R)$ is a gradient, the corresponding KMS-inequalities read $\norm{P}_{\lebe^{\infty}(\R)}\leq \norm{\mathscr{A}[P]}_{\lebe^{\infty}(\R)}$ for $P\in\hold_{c}^{\infty}(\R)$ and hold subject to the ellipticity assumption on $\mathscr{A}$. 
\end{remark}
These results only display a selection, and other scales such as smoothness spaces e.g. \`{a} la Triebel-Lizorkin \cite{Triebel} equally seem natural but are beyond the scope of this paper. However, returning to the above examples we single out the following  
\begin{OQ}[On weighted Lebesgue and Orlicz spaces]
\noindent
(a) If one aims to generalise the above result for weighted Lebesgue spaces to the case $p=1$ in $n\geq 3$ dimensions, a suitable weighted version of Lemma~\ref{lem:BoBre}~\ref{item:BoBre2} is required. To the best of our knowledge, the only available weighted Bourgain-Brezis estimates are due to \textsc{Loulit}~\cite{Loulit} but work subject to different assumptions on the weights than belonging to $A_{1}$. In this sense, it would be interesting to know whether~\eqref{eq:Xscale} extends to $X(\R^{n})=\lebe^{\frac{n}{n-1}}(\R^{n},w)$ and $Y=\lebe^{1}(\R^{n};w^{\frac{n-1}{n}})$ for $w\in A_{1}$. 

(b) The reader will notice that a slightly weaker bound can be obtained in the above Orlicz scenario provided one keeps the $\Delta_{2}$- and $\nabla_{2}$-assumptions on $B$ but weakens the assumptions on $A$. By~\textsc{Cianchi}~\cite[Thm.~2(i)]{CianchiFIT} the Riesz potential $\mathcal{I}_{1}$ maps $\lebe^{A}(\R^{n})\to\lebe_{\mathrm{w}}^{B}(\R^{n})$ (with the weak Orlicz space $\lebe_{\mathrm{w}}^{B}(\R^{n})$) precisely if~$\eqref{eq:CianchiMain}_{2}$ holds and $A^{(n)}$ dominates $B$ globally. For $A(t)=t$ and hereafter 
\begin{align*}
\widetilde{A}(t)=\begin{cases} 0&\;\text{if}\;0\leq t\leq 1,\\ 
\infty&\;\text{otherwise},
\end{cases}
\end{align*}
together with the $\Delta_{2}\cap\nabla_{2}$-function $B(t)=t^{\frac{n}{n-1}}$, this gives us the well-known mapping property $\mathcal{I}_{1}\colon \lebe^{1}(\R^{n})\to\lebe_{\mathrm{w}}^{\frac{n}{n-1}}(\R^{n})$ boundedly. Similarly as this boundedness property can be improved on divergence-free maps to yield estimate~\eqref{eq:BourgainBrezis1}, it would be interesting to know whether estimates of the form $\norm{u}_{\lebe^{B}(\R^{n})}\leq c \norm{\curl u}_{\lebe^{A}(\R^{n})}$ hold for $u\in\hold_{c,\div}^{\infty}(\R^{n};\R^{n})$ for if $n\geq 3$ provided~$\eqref{eq:CianchiMain}_{2}$ is in action (which is the case e.g. if $A$ has almost linear growth) and $A^{(n)}$ dominates the $\Delta_{2}\cap\nabla_{2}$-function $B$. This would imply~\eqref{eq:Xscale} under slightly weaker conditions on $A,B$ than displayed above. 
\end{OQ}

\section{Subcritical KMS-inequalities}\label{sec:subcritical}
We conclude the paper by giving the quick proof of Theorem~\ref{thm:main3}. As to Theorem~\ref{thm:main3}~\ref{item:qlessp1}, let $q=1$, $\varphi\in\hold_{c}^{\infty}(\R^{n};\R^{m})$ be arbitrary and choose $r>0$ so large such that $\spt(\varphi)\subset\ball_{r}(0)$. Since $c>0$ as in~\eqref{eq:KMS3} is independent of $r>0$, we may apply this inequality to $P=\D\varphi$. Sending $r\to\infty$, we obtain $\norm{\D\varphi}_{\lebe^{1}(\R^{n})}\leq c\norm{\A\varphi}_{\lebe^{1}(\R^{n})}$ for all $\varphi\in\hold_{c}^{\infty}(\R^{n};\R^{m})$. By the sharp version of \textsc{Ornstein}'s Non-Inequality as given by \textsc{Kirchheim \& Kristensen} \cite[Thm.~1.3]{KirchheimKristensen} (also see~\cite{ContiFaracoMaggi,FaracoGuerra}), the existence of a linear map $\mathrm{L}\in\mathscr{L}(\R^{N\times m};\R^{N\times m})$ with $\mathbb{E}_{i}=\mathrm{L}\circ\,\mathbb{A}_{i}$ for all $1\leq i\leq n$ follows at once. If $1<q<n$, we may similarly reduce to the situation on the entire $\R^{n}$ and use the argument given at the beginning of Section~\ref{sec:necpart} to find that ellipticity of $\A u\coloneqq \mathscr{A}[\D u]$ is necessary for~\eqref{eq:KMS3} to hold. 

For the sufficiency parts~of Theorem~\ref{thm:main3}~\ref{item:qlessp1} and~\ref{item:qlessp2}, we note that by scaling it suffices to consider the case $r=1$. Given $P\in\hold_{c}^{\infty}(\ball_{1}(0);\R^{m\times n})$, we argue as in Section~\ref{sec:sharpAllD} and Helmholtz decompose $P=P_{\Curl}+P_{\Div}$. Writing $P_{\Curl}=\D u$, we have $\norm{P_{\Curl}}_{\lebe^{q}(\ball_{1}(0))}\leq c\,(\norm{\mathscr{A}[P]}_{\lebe^{q}(\ball_{1}(0))}+\norm{P_{\Div}}_{\lebe^{q}(\ball_{1}(0))})$; for $q=1$ this is a trivial consequence of $\mathbb{E}_{i}=\mathrm{L}\circ\,\mathbb{A}_{i}$ for all $1\leq i\leq n$, whereas for $q>1$ this follows as in Section~\ref{sec:sharpAllD}, cf.~\eqref{eq:HelmholtzPart1}. Finally, the part $P_{\Div}$ is treated by use of the subcritical fractional integration theorem, cf.~Lemma~\ref{lem:FIT}~\ref{item:FItsubc}. This completes the proof of Theorem~\ref{thm:main3}. 

\begin{remark}[Variants and generalisations of Theorem~\ref{thm:main3}]
If $n=2$, Korn's inequality in the form of Theorem~\ref{thm:domains} (under the same assumptions on $\Omega\subset\R^{2}$)
\begin{align}\label{eq:kornerq}
\min_{\Pi\in\mathcal{K}}\norm{P-\Pi}_{\lebe^{q}(\Omega)}\leq c\,\Big(\norm{\mathscr{A}[P]}_{\lebe^{q}(\Omega)}+\norm{\Curl P}_{\lebe^{p}(\Omega)}\Big)
\end{align}
persists for all $1< p<2$ and all $1< q\leq p^{*}=\frac{2p}{2-p}$. One starts from~\eqref{eq:Necas4} and estimates 
\begin{align*}
\norm{\Curl P}_{\sobo^{-1,q}(\Omega)} & \leq c \sup_{\substack{\varphi\in\sobo_{0}^{1,q'}(\Omega;\R^{2})\\ \norm{\D\varphi}_{\lebe^{q'}(\Omega)}\leq 1}}\int_{\Omega}\langle\Curl P,\varphi\rangle\dif x \leq c\norm{\Curl P}_{\lebe^{p}(\Omega)} 
\end{align*}
by H\"{o}lder's inequality and $\norm{\varphi}_{\lebe^{p'}(\Omega)} \leq c\norm{\D\varphi}_{\lebe^{q'}(\Omega)}$ for $\varphi\in\sobo_{0}^{1,q'}(\Omega;\R^{2})$. If $q<2$ and thus $q'>2$, this follows by Morrey's embedding theorem and by the estimate $\norm{\varphi}_{\lebe^{s}(\Omega)}\leq c\norm{\D\varphi}_{\lebe^{2}(\Omega)}$ for all $1<s<\infty$ provided $q=2$. If $2<q\leq \frac{2p}{2-p}$, we require $p'\leq (q')^{*}$, and this is equivalent to $q\leq\frac{2p}{2-p}$. 
\end{remark}
We conclude with a link between the Orlicz scenario from Section~\ref{sec:sharpAllD} and Theorem~\ref{thm:main3}: 
\begin{remark}
As discussed by \textsc{Cianchi} et al. \cite{BCD,Cianchi} for the (trace-free) symmetric gradient, if the Young function $B$ is not of class $\Delta_{2}\cap\nabla_{2}$, then Korn-type inequalities persist with a certain loss of integrability on the left-hand side. For instance, by~\cite[Ex.~3.8]{Cianchi} one has for $\alpha\geq 0$ the inequality 
\begin{equation}
 \norm{\D u}_{\lebe\log^{\alpha}\lebe(\Omega)}\leq c\norm{\sym\D u}_{\lebe\log^{1+\alpha}\lebe(\Omega)} \quad \text{for all $u\in\hold_{c}^{\infty}(\Omega;\R^{n})$}
\end{equation}
 with the Zygmund classes $\mathrm{L}\log^{\alpha}\lebe$. By the method of proof, the underlying inequalities equally apply to general elliptic operators. One then obtains e.g. the refined subcritical KMS-type inequality 

\begin{align}\label{eq:ZygmundSave}
\norm{P}_{\lebe\log^{\alpha}\lebe(\ball_{1}(0))}\leq c\Big(\norm{\mathscr{A}[P]}_{\lebe\log^{1+\alpha}\lebe(\ball_{1}(0))}+\norm{\Curl P}_{\lebe^{p}(\ball_{1}(0))} \Big)
\end{align}
for all $P\in\hold_{c}^{\infty}(\ball_{1}(0);\R^{n\times n})$, where $\alpha\geq 0$, $p\geq 1$ and $c=c(p,\alpha,n,\mathscr{A})>0$ is a constant. The reader will notice that this inequality holds if and only if $\mathscr{A}$ induces an elliptic operator $\mathbb{A}$.  Note that for $\alpha=0$, the additional logarithm on the right-hand side of~\eqref{eq:ZygmundSave} is the key ingredient for~\eqref{eq:ZygmundSave} to hold for such $\mathscr{A}$, while \emph{without the additional logarithm} one is directly in the situation of Theorem~\ref{thm:main3}~\ref{item:qlessp1} and this not only forces $\mathbb{A}=\mathscr{A}[\D u]$ to be elliptic but to trivialise.
\end{remark}

{\footnotesize
\subsection*{Acknowledgment}
Franz Gmeineder gratefully acknowledges financial support by the Hector foundation and is moreover thankful for financial support by the University of Duisburg-Essen for a visit in April 2022, where parts of the project were concluded. Patrizio Neff is supported within the project 'A variational scale-dependent transition scheme - from Cauchy elasticity to the relaxed micromorphic continuum' (Project-ID 440935806). Moreover, Peter Lewintan and Patrizio Neff were supported in the Project-ID 415894848 by the Deutsche Forschungsgemeinschaft. The authors thank {the anonymous referee for his valuable comments and}  the staff of the University Library for their help in obtaining articles, especially older ones.
\vspace{-1.7ex}
\subsection*{Conflict of interest} The authors declare that they have no conflict of interest.
\vspace{-2.5ex}
\subsection*{Data availability statement} Data sharing not applicable to this article as no datasets were generated or analysed.}

  {\footnotesize
 \begin{alphasection}
 \section{Appendix}
In this appendix, we briefly revisit Helmholtz decompositions in arbitrary space dimensions, examine another class of part maps $\mathscr{A}$ (and hereafter differential operators $\A$) which is encountered in infinitesimal elasto-plasticity and gather some auxiliary results and terminology on weighted Lebesgue and Orlicz spaces as utilised in Setion~\ref{sec:sharpAllD}. 
\subsection{Explicit Helmholtz decomposition in $\R^{n}$}
Let us first recall the situation in the $3$-dimensional case. Let $a\in\R^3$ be a fixed vector. Since the cross product $a\times\cdot$ is linear in the second component it can be identified with a multiplication with the skew-symmetric matrix
\begin{equation}\label{eq:Anti}
 \Anti(a)=\begin{pmatrix} 0 & -a_3 & a_2 \\ a_3 & 0& -a_1\\ -a_2 & a_1 & 0  \end{pmatrix},
\end{equation}
so that
\begin{equation}
 a \times b = \Anti(a)\, b \qquad \text{for all $b\in\R^3$.}
\end{equation}
Hence, with $\Anti:\R^3\to \mathfrak{so}(3)\coloneqq\{A\in\R^{3\times 3}\mid A^\top=-A\}$ we have a canonical identification of vectors in $\R^3$ with skew-symmetric $(3\times3)$-matrices. It is this identification of the cross product with a suitable matrix multiplication that allows to generalise the usual cross product in $\R^3$ to a cross product of a vector $b\in\R^3$ and a matrix $P\in\R^{m\times 3}$, $m\in\N$, from the right
\begin{equation}\label{eq:crossProduct}
 P\times b \coloneqq P\,\Anti(b)
\end{equation}
which is seen as row-wise application of the usual cross product. In the Hamiltonian formalism the usual $\curl:\hold^\infty_c(\R^3;\R^3)\to \hold^\infty_c(\R^3;\R^3)$ expresses as
\begin{equation}
 \curl a = \nabla \times a = \Anti(\nabla)\, a \qquad \text{for all $a\in \hold^\infty_c(\R^3;\R^3)$.}
\end{equation}
Hence, it extends row-wise to a matrix-valued operator $\Curl$:
\begin{equation}\label{eq:matrixCurl}
 \Curl P \coloneqq P\times (-\nabla) = - P\,\Anti(\nabla) \qquad \text{for all $P\in \hold^\infty_c(\R^3;\R^{m\times 3})$}
\end{equation}
and the minus sign comes from the anti-commutativity of the cross product since we operate from the right hand side here. This gives rise to the definition of the generalised curl as displayed in Section \ref{sec:prelims}, and in particular in two dimensions it holds
\begin{equation}
 \curl \begin{pmatrix} a_1 \\ a_2  \end{pmatrix} = \partial_1 a_2 - \partial_2 a_1 = \div \begin{pmatrix} a_2 \\ -a_1  \end{pmatrix}\,.
\end{equation}
The usual Helmholtz decomposition in $n=3$ dimensions can then be deduced from the following decomposition of the vector Laplacian:
\begin{equation}\label{eq:decompoLaplace}
 \Delta a = \nabla \div a - \curl \curl a \qquad \text{for all $a\in \hold^\infty_c(\R^3;\R^3)$}
\end{equation}
and has the form
\begin{subequations}
\begin{align}
 a(x)&=-\frac{1}{4\pi} \nabla_x\int_{\R^3}\frac{\div a(y)}{|x-y|}\,\mathrm{d}y \ \ + \frac{1}{4\pi}\curl_x\int_{\R^3}\frac{\curl a(y)}{|x-y|}\,\mathrm{d}y \\
  & = \underset{=a_{\curl}(x)}{\underbrace{\frac{1}{4\pi}\int_{\R^3}\div a(y)\cdot\frac{x-y}{|x-y|^3}\,\mathrm{d}y}}\ \ \underset{=a_{\div}(x)}{\underbrace{- \frac{1}{4\pi}\int_{\R^3}\frac{x-y}{|x-y|^3}\times \curl a(y)\,\mathrm{d}y}}
\end{align}
\end{subequations}
with $\curl$-free part $a_{\curl}$ and divergence-free part $a_{\div}$.

Furthermore, as higher dimensional generalization of the decomposition of the vector Laplacian it holds (see \cite[Eq.~(4.28)]{Lewintan1})
\begin{equation}\label{eq:decompoLaplaceN}
 \Delta a = \nabla \div a + \nMat{\nabla}{n}^\top\curl a \qquad \text{for all $a\in \hold^\infty_c(\R^n;\R^n)$, $n\ge2$,}
\end{equation}
and the minus sign in the usual decomposition \eqref{eq:decompoLaplace} comes from the skew-symmetry of the associated matrices \eqref{eq:Anti}. Hence, following essentially the same steps to deduce the usual Helmholtz decomposition the relation \eqref{eq:decompoLaplaceN} allows to decompose any vector field $a\in \hold^\infty_c(\R^n;\R^n)$ into a divergence-free vector field $a_{\div}$ and a gradient field $a_{\curl}$, see \cite[Section 4.5]{Lewintan1} for more details, and we end up with an explicit coordinate-free expression of the Helmholtz decomposition in all dimensions with

\begin{align}\label{eq:HelmholtzN}
\begin{split}
 a_{\curl}(x) &= \nabla_x\int_{\R^n} \Phi(x-y)\cdot\div a(y)\,\mathrm{d}y =\frac{1}{n\,\omega_n} \int_{\R^n} (x-y)\cdot\left(\frac{\div a(y)}{|x-y|^{n}}\right)\,\mathrm{d}y, \\
 a_{\div}(x)&= \nMat{\nabla_x}{n}^\top\int_{\R^n} \Phi(x-y)\cdot\curl a(y)\,\mathrm{d}y  =\frac{1}{n\,\omega_n}\int_{\R^n} \nMat{x-y}{n}^\top\left(\frac{\curl a(y)}{|x-y|^{n}}\right)\,\mathrm{d}y,
 \end{split}
\end{align}
where $\Phi$ denotes the fundamental of the Laplacian in $n$ dimensions given by 
\begin{align*}
 \Phi(x)=\begin{cases}
     \frac{1}{2\pi}\log|x|, & \text{for } n=2, \\
     \frac{1}{n(2-n)\,\omega_n}|x|^{2-n}, &\text{for } n\ge3.
    \end{cases}
\end{align*}

\subsection{Connections to incompatible Maxwell-type inequalities and general terminology}\label{appendix:maxwell}
The inequalities studied in this paper are generalisations of well-known coercive estimates that have been employed e.g. in elasticity or general relativity. This also motivates the terminology of \emph{KMS-inequalities} used in the main part. As displayed in Figure~\ref{fig:neffpic}, the Cartan decomposition of a matrix field $P$ into its trace-free symmetric and its complementary part correspond to Korn-type inequalities and Maxwell-type (or $\div$-$\curl$-) inequalities for compatible fields. This is exemplarily displayed for the particular choice $\mathscr{A}[P]=\mathscr{A}_{\text{Korn}}[P]=\dev\sym P$ and contrasted by the results of the present paper in the incompatible case in Figure~\ref{fig:neffpic}. 

Both the Korn-type and Maxwell-type inequalities from Figure~\ref{fig:neffpic} in the compatible case can be reduced to the Calder\'{o}n-Zygmund-Korn Lemma~\ref{lem:CZK}; from a historic perspective, the Korn-type inequality from Figure~\ref{fig:neffpic} is due to~\textsc{Reshetnyak}~\cite{Reshetnyak} (also see~\textsc{Dain}~\cite{Dain}). In turn, the Maxwell- or $\div$-$\curl$-inequality is originally due to \textsc{Korn}~\cite[bottom of p. 707 for $(p,n)=(2,3)$]{Korn} and can be directly retrieved from Lemma~\ref{lem:CZK} by using the elliptic operator 
 $\A u = (\div u, \curl u)^\top$. Since the matrix divergence and matrix curl act \textit{row-wise} we further have for all $P\in \hold^{\infty}_c(\R^n;\R^{m\times n})$ with $n\ge2$, $m\in\N$ and $1<q<\infty$:
 \begin{equation}
  \norm{\D P}_{\lebe^q(\R^n)} \le c\,(\norm{\Div P}_{\lebe^q(\R^n)}+\norm{\Curl P}_{\lebe^q(\R^n)});
 \end{equation}
 as an especially important implication, one obtains by virtue of \textsc{Piola}'s identity $\Div\Cof \D u\equiv0$ that 
 \begin{equation}
\norm{\D \Cof \D u}_{\lebe^q(\R^n)} \le c\, \norm{\Curl \Cof \D u}_{\lebe^q(\R^n)}\qquad\text{for all}\;u\in\hold_{c}^{\infty}(\R^{n};\R^{n}).
\end{equation} 
Equivalently, the $\div$-$\curl$-inequality can also be directly deduced using classical elliptic regularity estimates and the following decomposition of the Laplacian ($n\ge2$):
\begin{equation}\label{eq:LaplacDecompo}
 \Delta u = \nabla \div u + L(\D\, \curl u),
\end{equation}
with a constant coefficient linear operator $L$, cf. \eqref{eq:decompoLaplaceN}; in our approach, the Maxwell-type inequality comes as a by-product.

\begin{figure}[t]
 \centering\includegraphics[width=\textwidth]{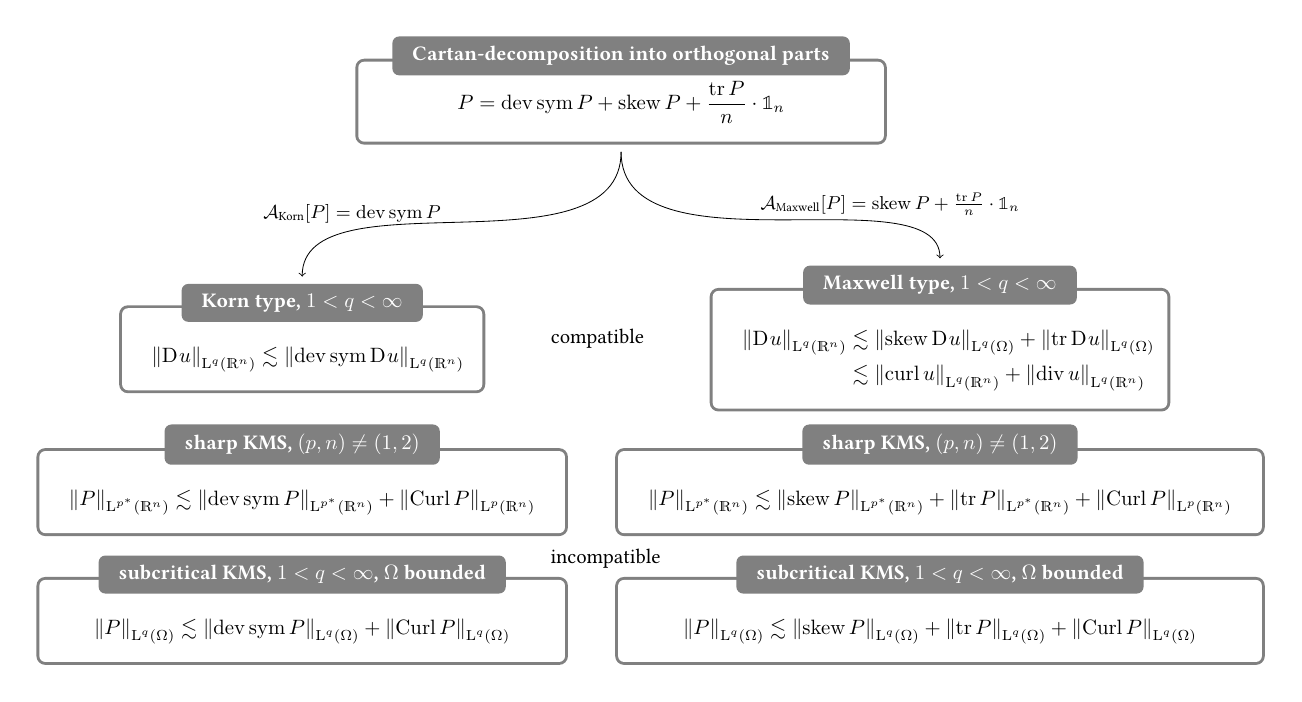}
\caption{\footnotesize The appearing vector valued functions $u$ and tensor valued functions $P$ are compactly supported.}\label{fig:neffpic}
\end{figure}

\subsection{Notions for weighted Lebesgue and Orlicz spaces}\label{sec:appendixweighted}
Here we gather some background results on weighted Lebesgue and Orlicz spaces as required for the generalisations of KMS-type inequalities in Section~\ref{sec:sharpAllD}. 

First, given a weight (i.e. a locally integrable non-negative function) $w\colon\R^{n}\to\R$ and $1<p<\infty$, we say that $w$ belongs to the \emph{Muckenhoupt class} $A_{p}$ provided 
\begin{align*}
[w]_{A_{p}}:=\sup_{\substack{Q\;\text{cube}}}\Big(\dashint_{Q}w\dif x \Big)\Big(\dashint_{Q}w^{1-p'}\dif x\Big)^{p-1} < \infty. 
\end{align*}
Moreover, denoting $\mathcal{M}$ the non-centered Hardy-Littlewood maximal operator, we say that $w$ belongs to the Muckenhoupt class $A_{1}$ provided there exists $c>0$ such that $\mathcal{M}w\leq cw$ holds $\mathscr{L}^{n}$-a.e. in $\R^{n}$. The smallest such constant $c>0$ then is denoted $[w]_{A_{1}}$. If $w\in A_{p}$ for some $1\leq p<\infty$, for $1\leq q<\infty$ the \emph{weighted Lebesgue space} $\lebe^{q}(\R^{n};w)$ then is the Lebesgue space with respect to the measure $\mu=w\mathscr{L}^{n}$. The Muckenhoupt condition $w\in A_{p}$ then yields boundedness of the usual Hardy-Littlewood maximal operators and singular integrals on $\lebe^{p}(\R^{n};w)$ provided $1<p<\infty$, cf.~\cite{Muckenhoupt1,Duoandi}, which is why Korn-type inequalities on $\lebe^{p}(\R^{n};w)$-spaces are available provided $w\in A_{p}$; to connect with the results from Section~\ref{sec:sharpAllD}, note that $A_{1}$ is the smallest Muckenhoupt class and is contained in any other $A_{p}$.  

Second, we say that a function $A\colon[0,\infty)\to[0,\infty]$ is a \emph{Young function} provided it can be written as 
\begin{align*}
A(s)=\int_{0}^{s}a(r)\dif r,\;\;\;r\geq 0, 
\end{align*}
for some increasing, left-continuous $a\colon[0,\infty)\to[0,\infty]$ such that $a$ does not equal zero or infinity on $(0,\infty)$. Each such Young function $A$ gives rise to the correspoding Orlicz space $\lebe^{A}(\R^{n})$ as the collection of all measurable $u\colon\R^{n}\to\R$ for which the Luxemburg norm 
\begin{align*}
\norm{u}_{\lebe^{A}(\R^{n})}:=\inf\Big\{\lambda>0\colon\;\int_{\R^{n}}A\Big(\frac{|u(x)|}{\lambda}\Big)\dif x \leq 1 \Big\}
\end{align*}
is finite. With a Young function $A$, we associate its \emph{Fenchel conjugate} via $\widetilde{A}(s):=\sup_{t\geq 0}(s\,t-A(t))$. We say that $A$ is of class $\Delta_{2}$ provided there exists $c>0$ such that $A(2s)\leq cA(s)$ holds for all $s\geq 0$, and of class $\nabla_{2}$ provided $\widetilde{A}$ is of class $\Delta_{2}$; exemplary instances of functions that satisfy both $\Delta_{2}$- and $\nabla_{2}$-conditions are given by the power functions $s\mapsto s^{p}$ for $1<p<\infty$. Similarly, given two Young functions $A,B$, we say that $A$ \emph{globally dominates} $B$ provided $B(s)\leq A(cs)$ holds for all $s\geq 0$ and a constant $c>0$. 

Following \textsc{Cianchi} \cite{CianchiFIT}, given a Young function $B$ and $1<p\leq\infty$, we denote $B_{(p)}$ the Young function with Fenchel conjugate
\begin{align}\label{eq:CianchiDef1}
\widetilde{B}_{(p)}(s)=\int_{0}^{s}r^{p'-1}(\Psi_{p}^{-1}(r^{p'}))^{p'}\dif r\;\;\;\text{with}\;\;\;\Psi_{p}(s):=\int_{0}^{s}\frac{B(t)}{t^{1+p'}}\dif t. 
\end{align}
Here, the inverse is understood in the usual generalised sense. Within the context that is of interest to us, we moreover define for a Young function $A$ and $1<p\leq\infty$
\begin{align}\label{eq:CianchiDef2}
A^{(p)}(s)=\int_{0}^{s}r^{p'-1}(\Phi_{p}^{-1}(r^{p'}))^{p'}\dif r\;\;\;\text{with}\;\;\;\Phi_{p}(s):=\int_{0}^{s}\frac{\widetilde{A}(t)}{t^{1+p'}}\dif t. 
\end{align}
These notions are instrumental in stating \textsc{Cianchi}'s sharp embedding result for Riesz potentials \cite[Thm.~2]{CianchiFIT} (also see our discussion on p.~\pageref{eq:CianchiMain}); also compare with \cite{BCD}.

\end{alphasection}
}
\end{document}